\documentclass{article}
\usepackage{amsfonts}
\usepackage{amsmath}
\usepackage{amssymb}
\usepackage{multirow}
\usepackage{lscape}
\input xy \xyoption{all}

\usepackage{rotating}
\usepackage{hyperref}
\usepackage{color}

\newtheorem{theorem}{Theorem}[section]

\newtheorem{corollary}[theorem]{Corollary}

\newtheorem{definition}[theorem]{Definition}

\newtheorem{lemma}[theorem]{Lemma}

\newtheorem{proposition}[theorem]{Proposition}

\newtheorem{remark}[theorem]{Remark}
\newenvironment{proof}[1][Proof]{\noindent\textbf{#1.} }{\ \rule{0.5em}{0.5em}}

\newcommand{\f}{\mathfrak}

\newcommand{\wnabla}{\widetilde{\nabla}}
\newcommand{\bb}{\mathbb}
\newcommand{\R}{\mathbb{R}}
\newcommand{\C}{\mathbb{C}}
\newcommand{\QK}{\mathcal{QK}}
\newcommand{\HK}{\mathcal{HK}}

\font\frd=eufm10 scaled\magstep3
\newcommand{\Cyclic}[2]{\genfrac{}{}{0pt}{0}{#1}{#2}}

\newcommand{\fdS}{\raisebox{-2.0ex}{\mbox{\frd{S}}}}
\newcommand{\SXYZ}{\raisebox{-0.6ex}{\mbox{\scriptsize{$\mathit{XYZ}$}}}}

\newcommand{\SJXJYJZ}{\raisebox{-0.6ex}{\mbox{\scriptsize{$\mathit{JXJYJZ}$}}}}

\newcommand\T{\rule{0pt}{3.1ex}}
\newcommand\B{\rule[-1.7ex]{0pt}{0pt}}

\author{M. Castrill\'on L\'opez \\
ICMAT (CSIC-UAM-UC3M-UCM)\\
Departamento de Geometr\'\i a y Topolog\'\i a \\
Facultad de Matem\'aticas, Universidad Complutense de Madrid\\
28040 Madrid, Spain \and Ignacio Luj\'an  \\Departamento de Geometr\'\i a y Topolog\'\i a \\
Facultad de Matem\'aticas, Universidad Complutense de Madrid\\
28040 Madrid, Spain}
\date{}

\title{Strongly degenerate homogeneous pseudo-K\"ahler structures of linear type and complex plane waves}

\begin{document}

\author{M. Castrill\'on L\'opez \\
\small{ICMAT (CSIC-UAM-UC3M-UCM)}\\
\small{Departamento de Geometr\'\i a y Topolog\'\i a} \\
\small{Facultad de Matem\'aticas, Universidad Complutense de Madrid}\\
\small{28040 Madrid, Spain} \and Ignacio Luj\'an  \\\small{Departamento de Geometr\'\i a y Topolog\'\i a} \\
\small{Facultad de Matem\'aticas, Universidad Complutense de Madrid}\\
\small{28040 Madrid, Spain}}
\date{}

\title{Strongly degenerate homogeneous pseudo-K\"ahler structures of linear type and complex plane waves}

\maketitle

\abstract{We study the class $\mathcal{K}_{2}+\mathcal{K}_{4}$ of
homogeneous pseudo-K\"ahler structures in the strongly degenerate
case. The local form and the holonomy of a pseudo-K\"ahler
manifold admitting such a structure is obtained, leading to a
possible complex generalization of homogeneous plane waves. The
same question is tackled in the case of pseudo-hyper-K\"ahler and
pseudo-quaternion K\"ahler manifolds.}

\renewcommand{\thefootnote}{\fnsymbol{footnote}}
\renewcommand{\thefootnote}{\arabic{footnote}}
\renewcommand{\thefootnote}{\fnsymbol{footnote}}
\footnotetext{K\"ahler,
pseudo-quaternion K\"ahler, reductive homogeneous
pseudo-Riemannian spaces.}
\renewcommand{\thefootnote}{\arabic{footnote}}

\section{Introduction}

Undoubtedly, homogeneous manifolds constitute a distinguished
class of spaces on which the study of (pseudo)-Riemannian geometry
is especially rich and varied. They enjoy a privileged position in
Differential Geometry and have been intensively studied by means
of a fruitful collection of approaches and tools. Among them,
homogeneous structure tensors have proved to be one of the most
successful. These tensors combine their algebraic structure
together with a set of geometric PDEs known as Ambrose-Singer
equations (see \cite{AS}, \cite{TV}). In addition, homogeneous
spaces play an essential role in different contexts of theoretical
Physics as in Field Theories (cf. \cite{Cec}, \cite{dWVP}),
Gravitation (cf. \cite{DK}, \cite{OOF}), etc.

From a geometrical point of view, homogeneous structures have been
able to characterize certain spaces. For definite metrics, purely
Riemannian or with additional geometry (K\"ahler or quaternion
K\"ahler for instance), the so-called homogeneous structures of
linear type characterize negative constant sectional (holomorphic
sectional, quaternionic sectional) curvature (see \cite{CGS0},
\cite{GMM} and \cite{TV} for these results as well as indications
of other similar results). When the case of metrics with signature
is analyzed, the causal nature of the vector fields characterizing
a homogeneous structure of linear type gives rise to different
scenarios with some physical implications. In \cite{Mes} and
\cite{Mon} the purely pseudo-Riemannian case with an isotropic
structure is studied in full detail: More precisely, it is proved
that these spaces have the underlying geometry of a real singular
homogeneous plane wave. The aim of this paper is to extend this
result to the pseudo-K\"ahler, pseudo-hyper-K\"ahler and
pseudo-quaternion K\"ahler settings.

The main characterization of this work gives the geometry of
pseudo-K\"ahler manifolds with a so called strongly degenerate
homogeneous pseudo-K\"ahler structure of linear type. In
particular, the expression of the metric obtained in the
characterization has strong similarities with singular
scale-invariant homogeneous plane waves. Furthermore, the
manifolds under study and these homogeneous plane waves share some
other analogies and features as it is shown in \S \ref{secwave}.
Because of this, since there is not a formal definition of
``complex plane wave'' (as far as the authors know),
pseudo-K\"ahler manifolds with strongly degenerate linear
homogeneous structures seem to be the correct generalization of
this special kind of homogeneous plane waves in complex framework,
at least in the important particular K\"ahler case. In addition,
the same techniques are successfully applied to a comparison of
Cahen-Wallach spaces and one of the possible pseudo-K\"ahler
symmetric spaces of index $2$ in the classification given in
\cite{Gal}, giving a more general picture of complex plane waves.

A relevant fact about these spaces is that they have holonomy
group contained in $SU(p,q)$. Consequently, it is natural to study
strongly degenerate homogeneous pseudo-hyper-K\"ahler and
pseudo-quaternion K\"ahler structures of linear type. However, we
prove that a manifold admitting any of those structures is
necessarily flat, pointing out that the notion of homogeneous
plane wave can not be realized in the pseudo-hyper-k\"ahler or
pseudo-quaternion K\"ahler cases in a non-trivial way.

The paper is organized as follows. In Section $2$ we recall some
results concerning homogeneous pseudo-Riemannian structures and
give some definitions. In Section $3$ the main result of the
article gives the local form of the metric of a pseudo-K\"ahler
manifold admitting a strongly degenerate homogeneous
pseudo-K\"ahler structure of linear type. The holonomy algebra and
some geometric properties are given. In addition, we analyze
$\mathbb{C}^{n+2}$ with the metric above as the local model of
these manifolds. In particular, the geodesic completeness is
studied showing the existence of cosmological singularities. In
Section $4$ we study the homogeneous model $G/H$ associated to
these kind of homogeneous structures. Geodesic completeness is
again analyzed. In Section $5$ we restrict ourselves to the
Lorentz-K\"ahler case (signature $(2,2+2n)$). We exhibit the
relationship between pseudo-K\"ahler manifolds with a strongly
degenerate structure of linear type and one of the possible
pseudo-K\"ahler symmetric spaces of index $2$ on one hand, and
singular scale-invariant homogeneous plane waves and Cahen Wallach
spaces on the other. In Section $6$ the pseudo-hyper-K\"ahler and
pseudo-quaternion K\"ahler cases are studied.


\section{Preliminaries}

\begin{definition}
A pseudo-Riemannian manifold $(M,g)$ is called homogeneous if
there is a connected Lie group of isometries $G$ acting
transitively on $M$. In this case, $(M,g)$ is called reductive
homogeneous if the Lie algebra $\f{g}$ of $G$ can be decomposed as
$\f{g}=\f{m}\oplus\f{h}$, where $\f{h}$ is
 the isotropy algebra of a point $p\in M$, and $[\f{m},\f{h}]\subset\f{m}$.
\end{definition}

In \cite{GO}, Ambrose-Singer Theorem \cite{AS} is extended to the
pseudo-Rie- mannian setting:

\begin{theorem}
Let $(M,g)$ be a connected, simply-connected and (geodesically) complete pseudo-Riemannian manifold. It is equivalent:
\begin{enumerate}
\item $(M,g)$ is reductive homogeneous.

\item $(M,g)$ admits a linear connection $\wnabla$ such that
\begin{equation}\label{AS}\wnabla g=0,\qquad \wnabla R=0,\qquad \wnabla S=0,\end{equation}
where $S=\nabla-\wnabla$, $\nabla$ is the Levi-Civita connection of $g$, and $R$ is the curvature tensor field of $g$.
\end{enumerate}
\end{theorem}
A $(1,2)$-tensor field $S$ satisfying equations (\ref{AS}) (called
Ambrose-Singer equations) is called a \textit{homogeneous
pseudo-Riemannian
 structure}. We will also denote by $S$ the associated
$(0,3)$-tensor field obtained by lowering the contravariant index,
$S_{XYZ}=g(S_XY,Z)$. Let $\mathcal{S}$ be the space of homogeneous
pseudo-Riemannian structures, it is decomposed in three primitive
classes
\begin{align*}
\mathcal{S}_1 &  = \Bigl\{ S \in \mathcal{S} \,:\,S_{XYZ}=g(X,Y)\theta(Z)-g(X,Z)\theta(Y), \,\theta\in\Gamma(T^*M)\Bigr\},\\
\mathcal{S}_2 & = \Bigl\{ S \in \mathcal{S} \,:\,\Cyclic{\fdS}{\SXYZ}S_{XYZ}=0, \,c_{12}(S)=0 \Bigr\},\\
\mathcal{S}_3 & = \Bigl\{ S \in \mathcal{S} \,:\,S_{XYZ}+S_{YXZ}=0 \Bigr\},\\
\end{align*}
where $c_{12}(S)(Z)=\sum_{i=1}^{m}\epsilon_iS_{e_ie_iZ}$ for any
orthonormal basis $\{e_1,\ldots,e_{m}\}$ with
$g(e_i,e_i)=\epsilon_i$. A tensor field in the class
$\mathcal{S}_1$ is called of linear type, and in addition, it is
called degenerate if the vector field $\theta^{\sharp}$ is
isotropic.

\bigskip

Let $(M,g)$ be a pseudo-Riemannian manifold of dimension $2n$, and $J$ a pseudo-K\"ahler structure, that is, a parallel pseudo-Hermitian structure with respect to the Levi-Civita connection.

\begin{definition}
 A pseudo-K\"ahler manifold $(M,g,J)$ is called a homogeneous pseudo-K\"ahler manifold
 if there is a connected Lie group of isometries $G$ acting transitively on $M$ and preserving
 $J$. In this case $(M,g,J)$ is called reductive homogeneous pseudo-K\"ahler if the Lie algebra $\f{g}$ of $G$
 can be decomposed as $\f{g}=\f{m}\oplus\f{h}$, where $\f{h}$ is the isotropy algebra of a point $p\in M$, and $[\f{m},\f{h}]\subset\f{m}$.
\end{definition}

As a corollary of Kiri\v{c}enko's theorem  \cite{Kir} we have.

\begin{theorem}
Let $(M,g,J)$ be a connected, simply-connected and (geodesically) complete pseudo-K\"ahler manifold. It is equivalent:
\begin{enumerate}
\item $(M,g,J)$ is reductive homogeneous pseudo-K\"ahler.

\item $(M,g,J)$ admits a linear connection $\wnabla$ such that
\begin{equation}\label{AS+J}\wnabla g=0,\qquad \wnabla R=0,\qquad \wnabla S=0,\qquad \wnabla J=0,\end{equation}
where $S=\nabla-\wnabla$, $\nabla$ is the Levi-Civita connection of $g$, and $R$ is the curvature tensor field of $g$.
\end{enumerate}
\end{theorem}

A $(1,2)$ tensor field $S$ satisfying equations (\ref{AS+J}) is called a \textit{homogeneous pseudo-K\"ahler structure}.

In \cite{BGO}, a classification of homogeneous pseudo-K\"ahler
structures is obtained. Let $\mathcal{K}$ be the space of
homogeneous pseudo-K\"ahler structures, it is decomposed in four
primitive classes

\begin{align*}
\mathcal{K}_1 &  = \Bigl\{ S \in \mathcal{K} \,:\,S_{XYZ}=\frac{1}{2}(S_{YZX}+S_{ZXY}+S_{JYJZX}+S_{JZXJY}),\\ & \hspace{81mm}\,c_{12}(S)=0 \Bigr\},\\
\mathcal{K}_2 & = \Bigl\{ S \in \mathcal{K} \,:\,S_{XYZ}=g( X,Y )\theta_1(Z)-g( X,Z)\theta_1(Y)+g( X,JY )\theta_1(JZ)\\
& \hspace{21mm}-g( X,JZ)\theta_1(JY)-2g(JY,Z)\theta_1(JX), \,\theta_1\in\Gamma(T^*M)\Bigr\},\\
\mathcal{K}_3 & = \Bigl\{ S \in \mathcal{K} \,:\,S_{XYZ}=-\frac{1}{2}(S_{YZX}+S_{ZXY}+S_{JYJZX}+S_{JZXJY}),\\ & \hspace{81mm}\,c_{12}(S)=0 \Bigr\},\\
\mathcal{K}_4 & = \Bigl\{ S \in \mathcal{K} \,:\,S_{XYZ}=g( X,Y )\theta_2(Z)-g( X,Z)\theta_2(Y)+g( X,JY )\theta_2(JZ)\\
& \hspace{21mm}-g( X,JZ)\theta_2(JY)+2g(JY,Z)\theta_2(JX), \,\theta_2\in\Gamma(T^*M)\Bigr\},
\end{align*}
where $c_{12}(S)(Z)=\sum_{i=1}^{2n}\epsilon_iS_{e_ie_iZ}$ for any orthonormal basis $\{e_1,\ldots,e_{2n}\}$ with $g(e_i,e_i)=\epsilon_i$.

\bigskip

A homogeneous pseudo-K\"ahler structure on a pseudo-K\"ahler
manifold $(M,g,J)$ is called of \textit{linear type} if it belongs to the
class $\mathcal{K}_2+\mathcal{K}_4$. The expression of a tensor field $S$ in this class is
$$S_XY=g(X,Y)\xi-g(Y,\xi)X-g(X,JY)J\xi+g(JY,\xi)JX-2g(JX,\zeta)JY,$$
for some vector fields $\xi,\zeta\in\f{X}(M)$, and equations
\eqref{AS+J}  are equivalent to
$$\wnabla \xi=\wnabla \zeta=0,\qquad \wnabla R=0.$$
With respect to the last condition, a simple computation shows
that $\wnabla R$ is independent of $\zeta$ so that $\xi =0$
implies $\nabla R=0$, that is, $M$ is locally symmetric. Along
this article, we confine ourselves to the case $\xi \neq 0$. For
non-definite metrics, we thus may distinguish the following cases:

\begin{definition}
A homogeneous pseudo-K\"ahler structure of linear type is called:
\begin{enumerate}
\item  non-degenerate if $g(\xi,\xi)\neq 0$. \item weakly
degenerate if $\xi \neq 0$ is isotropic and
$\zeta\in\mathrm{span}\{\xi,J\xi\}^{\bot}$. \item degenerate if
$\xi \neq 0$ and $\zeta$ are isotropic and
$\zeta\in\mathrm{span}\{\xi,J\xi\}^{\bot}$. \item
strongly degenerate if $\xi \neq 0$ is isotropic and $\zeta=0$.
\end{enumerate}
\end{definition}

\section{Strongly degenerate homogeneous structures of linear type}

We confine ourselves to the strongly degenerate case, that is
\begin{equation}
\label{degenerate}
S_XY=g(X,Y)\xi-g(Y,\xi)X-g(X,JY)J\xi+g(JY,\xi)JX.
\end{equation}
In \cite{BGO} the following is proved.

\begin{lemma}\label{Lemma BGO}
Let $(M,g,J)$ be a connected pseudo-K\"ahler manifold of dimension
$2n+4$, $n\geq0$, equipped with a strongly degenerate homogeneous
pseudo-K\"ahler structure $S$ of linear type. Let $\theta$ be the
$1$-form given by $\theta(X)=g(X,\xi)$. Then
\begin{enumerate}
\item $\nabla\theta=\theta\otimes\theta-(\theta\circ
J)\otimes(\theta\circ J)$. \item $\f{S}_{XYZ}\theta(X)R_{YZWU}=0$
for all $X,Y,Z,W,U\in\f{X}(M)$. \item $\nabla R=4\theta\otimes R$,
i.e., the manifold is recurrent (and hence harmonic).
\end{enumerate}

\end{lemma}

\bigskip

Note that the anti-symmetrization of the first equation gives that
$d\theta=0$, so $\theta$ is closed.  Also note that the second
equation can be written as
\begin{equation} \label{thR}
\theta\wedge R(\cdot,\cdot,W,U)=0.
\end{equation}
Changing $X,Y,Z$ by $JX,JY,JZ$ we will have that
$$\Cyclic{\fdS}{\SJXJYJZ}\theta(JX)R_{JYJZWU}=0,$$ but since $R$ is the
curvature of a pseudo-K\"ahler metric, this equation can be
written as
\begin{equation} \label{thRJ}
(\theta \circ J)\wedge R(\cdot,\cdot,W,U)=0.
\end{equation}

We consider the complex form
$$\alpha=\theta-i(\theta\circ J),$$
which is of type $(1,0)$ with respect to $J$. By direct
calculation one has that $\nabla \alpha=\alpha\otimes\alpha$, so
again by anti-symmetrization one obtains $d\alpha=0$, and in
particular $\overline{\partial}\alpha$=0. Then, fixing a point
$p\in M$, there exists a neighborhood $U$ around $p$ and a
function $v:U\to\mathbb{C}$ such that $dv=\alpha$. Since $\alpha$
is of type $(1,0)$ and $\alpha=dv=\partial
v+\overline{\partial}v$, it must be $\overline{\partial}v=0$, so
$v$ is holomorphic. Let $w:U\to\bb{C}$, $w=e^{-v}$, then
\begin{equation}\label{nabla dw}
\nabla dw=\nabla(-e^{-v}dv)=-d(e^{-v})\otimes dv-e^{-v}\nabla
dv=e^{-v}\alpha\otimes\alpha-e^{-v}\alpha\otimes\alpha=0.
\end{equation}

We write $v=v^1+iv^2$ and $w=w^1+iw^2$ so that
$w_1=e^{-v^1}\cos{v^2}$ and $w^2=-e^{-v^1}\sin{v^2}$, and as
$dv^1=\theta$ and $dv^2=-\theta\circ J$ we have
\begin{equation}\label{equatio dw1 dw2}
\left\{ \begin{array}{l} dw^1=-(e^{-v^1}\cos{v^2})\theta +
(e^{-v^1}\sin{v^2})(\theta\circ
J)\\
dw^2=(e^{-v^1}\sin{v^2})\theta + (e^{-v^1}\cos{v^2})(\theta\circ
J).
\end{array} \right.
\end{equation}
From (\ref{thR}), \eqref{thRJ} and the last two equations we have
\begin{equation}\label{eq dw wedge R}
\left\{\begin{array}{c} dw^1\wedge R(\cdot,\cdot,W,U)=0\\
dw^2\wedge R(\cdot,\cdot,W,U)=0

\end{array}\qquad W,U\in\f{X}(M).\right.
\end{equation}

Let now $e_{w_1}, e_{w_2}$ be vector fields such that
$dw^i(e_{w_j})=\delta_{j}^i$, from (\ref{eq dw wedge R}) we have
for $i=1,2$
\begin{equation}
\label{RRR} 0=i_{e_{w_i}}(dw^i\wedge
R(\cdot,\cdot,W,U))=R(\cdot,\cdot,W,U)-dw^i\wedge
R(e_{w_i},\cdot,W,U), \end{equation}
that is
$R(\cdot,\cdot,W,U)=dw^i\wedge R(e_{w_i},\cdot,W,U)$. Applying
this we have
\begin{eqnarray*}R(e_{w_i},\cdot,W,U) & =
& R(W,U,e_{w_i},\cdot)\\ & = & W(w^i)R(e_{w_i},U,e_{w_i},\cdot)-U(w^i)R(e_{w_i},W,e_{w_i},\cdot)
\end{eqnarray*}
and substituting in \eqref{RRR} we obtain
\begin{eqnarray*}
R(X,Y,W,U) & = & X(w^i)W(w^i)R(e_{w_i},U,e_{w_i},Y)\\
           & & -X(w^i)U(w^i)R(e_{w_i},W,e_{w_i},Y)\\
           &   & -Y(w^i)W(w^i)R(e_{w_i},U,e_{w_i},X)\\
           & & +Y(w^i)U(w^i)R(e_{w_i},W,e_{w_i},X),
\end{eqnarray*}
i.e.,
$$R=(dw^i\otimes dw^i)\wedge R(e_{w_i},\cdot,e_{w_i},\cdot).$$
Since this is valid for $i=1,2$, if we take a basis
$\{e_{w_1},e_{w_2},e_a\}_{a=1,\ldots,2n+2}$ of $T_qM$, $q\in U$,
such that $dw^i(e_a)=0$, it follows that all the components of $R$
with respect to this basis other than
$R(e_{w_1},e_{w_2},e_{w_1},e_{w_2})$ vanish.

\subsection{The local form of the metric}

\begin{theorem}\label{prop princ}
Let $(M,g,J)$ be a pseudo-K\"ahler manifold of dimension $2n+4$,
$n\geq 0$, admitting a strongly degenerate homogeneous
pseudo-K\"ahler structure  of linear type $S$. Then each $p\in M$
has a neighborhood holomorphically isometric to an open subset of
$\bb{C}^{n+2}$ with the K\"ahler metric
\begin{multline}\label{metric2}
g  =  dw^1dz^1+dw^2dz^2+b(dw^1dw^1+dw^2dw^2)+\sum_{a=1}^nr_a(dx^adw^1+dy^adw^2)\\
    +\sum_{a=1}^ns_a(dx^adw^2-dy^adw^1)+\sum_{a=1}^n(dx^adx^a+dy^ady^a),
\end{multline}
where $\epsilon_a=\pm 1$, and the functions $b,r_a,s_a$ only
depend on the coordinates $\{w^1,w^2\}$ and satisfy
\begin{equation}\label{3
equations singular}\dfrac{\partial s_a}{\partial
w^1}=\dfrac{\partial r_a}{\partial w^2},\qquad \dfrac{\partial
s_a}{\partial w^2}=-\dfrac{\partial r_a}{\partial
w^1}\vspace{2mm},\qquad \Delta
b=\dfrac{b_0}{\left((w^1)^2+(w^2)^2\right)^2},\end{equation} for
$b_0\in\R$ and $a=1,\ldots,n$.
\end{theorem}

\begin{proof}
Let $p\in M$ and $w=w^1+iw^2:U\to\bb{C}$ be the holomorphic
function obtained above. Since $dw$ is not zero at some point and
$\nabla dw=0$ we have that $dw$ is nowhere zero. Then for every
$\lambda\in\bb{C}$, if the pre-image
$\mathcal{H}_{\lambda}=w^{-1}(\lambda)$ is nonempty it defines a
regular complex hypersurface whose tangent space is given by the
kernel of $dw$. This means that we have our neighborhood $U$
foliated by complex hypersurfaces $\mathcal{H}_{\lambda}$ for
$\lambda$ in some open set of $\bb{C}$ (note that if
$w(p)=\lambda_0$ then the set
$\{\lambda\in\bb{C}/w^{-1}(\lambda)\neq\emptyset\}$ is a
neighborhood of $\lambda_0$ not containing $0\in\bb{C}$).

Let $Z,JZ\in\f{X}(U)$ be such that $dw^1=g(\cdot,Z)$ and
$dw^2=g(\cdot,JZ)$ respectively. By the expression of $dw^1$ and
$dw^2$ in terms of $\theta$ and $\theta\circ J$ it is easy to see
that both $Z$ and $JZ$ are  linear combinations of $\xi$ and
$J\xi$, so they are isotropic. In addition as $Z$ and $JZ$ are
orthogonal we have $dw(Z)=dw(JZ)=0$. This means that $Z$ and $JZ$
are always tangent to the foliation given by the hypersurfaces
$\mathcal{H}_{\lambda}$. On the other hand, from \eqref{nabla dw}
we have
$$\nabla Z=\nabla JZ=0.$$
In particular, if $\mathcal{L}$ denotes the Lie derivative, we
have
\begin{eqnarray*}
(\mathcal{L}_Zg)(X,Y)& = & Zg(X,Y)-g([Z,X],Y)-g(X,[Y,Z])\\
                     & = & Zg(X,Y)-g(\nabla_ZX,Y)+g(\nabla_XZ,Y)\\
                     &   & -g(X,\nabla_ZY)+g(X,\nabla_YZ)\\
                     & = & 0,
\end{eqnarray*}
and
\begin{eqnarray*}
(\mathcal{L}_ZJ)X & = & [Z,JX]-J[Z,X]\\
                  & = & \nabla_ZJX-\nabla_{JX}Z-J\nabla_ZX+J\nabla_XZ\\
                  & = & 0,
\end{eqnarray*}
and the same for $JZ$. This means that $Z$ and $JZ$ are Killing
holomorphic vector fields. Moreover,
$[Z,JZ]=\nabla_ZJZ-\nabla_{JZ}Z=0$. Reducing the neighborhood $U$
if necessarily, we can thus take complex coordinates
$\{w,z,\tilde{z}^a\}_{a=1,\ldots,n}$ such that $w$ is the function
$w=w^1+iw^2$, $Z=\partial_{z^1}$ and $JZ=\partial_{z^2}$ where
$z=z^1+iz^2$, and $\{z,\tilde{z}^a\}$ are coordinates adapted to
the foliation given by $\mathcal{H}_{\lambda}$.

Let $\tilde{z}^a=\tilde{x}^a+i\tilde{y}^a$, it is obvious by
definition that $g(\partial_{w^i},\partial_{z^j})=\delta^i_j$.
Also since $\partial_{\tilde{x}^a}$ and $\partial_{\tilde{y}^a}$
are tangent to the foliation $\mathcal{H}_{\lambda}$, we have
$g(\partial_{z^i},\partial_{\tilde{x}^a})=g(\partial_{z^i},\partial_{\tilde{y}^a})=0$.
Hence the metric in the real coordinates
$\{w^1,w^2,z^1,z^2,\tilde{x}^a,\tilde{y}^a\}$ is given by
\begin{multline*}
g  =  dw^1dz^1+dw^2dz^2+b(dw^1dw^1+dw^2dw^2)\\
  +\sum_{a=1}^n\tilde{r}_a(d\tilde{x}^adw^1+d\tilde{y}^adw^2)
  +\sum_{a=1}^n\tilde{s}_a(d\tilde{x}^adw^2-d\tilde{y}^adw^1)+q(w^1,w^2),
\end{multline*}
where $q(w^1,w^2)$ is an Hermitian metric in the coordinates
$\{\tilde{x}^a,\tilde{y}^a\}$, and the functions
$b,\tilde{r}_a,\tilde{s}_a$ do not depend on the coordinates
$z^1,z^2$ as $Z$ and $JZ$ are Killing vector fields.

Among all the possible choices of the coordinates
$\{\tilde{x}^a,\tilde{y}^a\}$ along the leaves of the foliation
$\mathcal{H}_\lambda$ we now construct a convenient set to
simplify the term $q(w^1,w^2)$ above. Recall that all the
components of the curvature (as a $(0,4)$-tensor) other than
$R(\partial_{w^1},\partial_{w^2},\partial_{w^1},\partial_{w^2})$
vanish. With respect to the curvature $(1,3)$-tensor field we
obtain
\begin{eqnarray*}
R=R(\partial_{w^1},\partial_{w^2},\partial_{w^1},\partial_{w^2})\left\{(dw^1\wedge
dw^2)\otimes(dw^1\otimes\partial_{z^2})\right.\\
\left. -(dw^1\wedge
dw^2)\otimes(dw^2\otimes\partial_{z^1})\right\}.\end{eqnarray*}
This means that the endomorphism $R_{XY}:T_pM\to T_pM$ is zero if
$X,Y\notin\mathrm{span}\{\partial_{w^1},\partial_{w^2}\}$ and
\begin{equation}\label{curvature op}\begin{array}{rrcl}
R_{\partial_{w^1}\partial_{w^2}}: & T_pM & \to & T_pM\\
                                   & \partial_{w^1} & \mapsto &
                                   R(\partial_{w^1},\partial_{w^2},\partial_{w^1},\partial_{w^2})\partial_{z^2}\\
                                   & \partial_{w^2} & \mapsto &
                                   -R(\partial_{w^1},\partial_{w^2},\partial_{w^1},\partial_{w^2})\partial_{z^1}\\
                                   & W & \mapsto & 0 \hspace{3em}\text{if} \hspace{1em} W\notin
                                   \mathrm{span}\{\partial_{w^1},\partial_{w^2}\}.
\end{array}\end{equation}
It is a well-known result (cf. \cite[Ch.3,\S 9]{KN}) that the
holonomy algebra $\f{hol}_p$ of $\nabla$ is spanned by all the
elements
$$\tau^{-1}\circ R_{\tau X\tau Y}\circ \tau$$
where $X,Y\in T_pM$ and $\tau$ is the parallel displacement along
an arbitrary piecewise differentiable curve starting at $p$. We
restrict ourselves to the holonomy of the open set $(U,g)$. Let
$E\subset T_pM$ be the subspace spanned by
$\{\partial_{w^1},\partial_{w^2},\partial_{z^1},\partial_{z^2}\}$
at $p$. Since the image of the curvature operator is contained in
$\mathrm{span}\{\partial_{z^1},\partial_{z^2}\}$ and these are
invariant by parallel displacement, we have that $E$ is invariant
by the action of the holonomy algebra $\f{hol}_p (U,g)$. Taking
$U$ simply-connected, this implies that $E$ is invariant by the
holonomy group $Hol_p(U,g)$. Since $E^{\bot}$ is also invariant by
the action of the holonomy and the image of the curvature operator
is contained in $E$, the action on $E^{\bot}$ is necessarily
trivial. Then, let $\{(e_1)_p,(Je_1)_p,\ldots,(e_n)_p,(Je_n)_p\}$
be an orthonormal basis of $E^{\bot}$, we can extend the vectors
$(e_a)_p$ and $(Je_a)_p$ along $U$ by parallel transport
independently of the chosen path. In this way we obtain parallel
vector fields $e_1,Je_1,\ldots,e_n,Je_n$ on $U$. Note that since
$\nabla$ is torsionless the condition $\nabla e_a=0=\nabla Je_a$
implies that all these vector fields commute between them and with
$\partial_{z^1}$ and $\partial_{z^2}$. Besides, it implies that
$\mathcal{L}_{e_a}g=0=\mathcal{L}_{Je_a}g$ and
$\mathcal{L}_{e_a}J=0=\mathcal{L}_{Je_a}J$, so they are Killing
holomorphic vector fields.

Let $X$ be one of the vector fields
$e_a,Je_a,\partial_{z^1},\partial_{z^2}$, and let $\gamma$ be any
path starting at $p$. We have that
$$\frac{d}{dt} (dw(X)_{\gamma(t)})=(\nabla_{\dot{\gamma}(t)} dw)(X_{\gamma(t)})+d w\left(\nabla_{\dot{\gamma}(t)}X\right)=0,$$
so the function $dw(X)$ is constant in $U$, and since $dw(X)_p=0$,
it is identically zero. This means that the vector fields
$e_a,Je_a,\partial_{z^1},\partial_{z_2}$ are tangent to the
hypersurfaces $\mathcal{H}_{\lambda}$ so their flows starting at a
point in a hypersurface $\mathcal{H}_{\lambda}$ remain in
$\mathcal{H}_{\lambda}$.

This means that we can take complex coordinates $\{w,z,z^a\}$,
$a=1,\ldots,n$, such that $\{z,z^a\}$, $a=1,\ldots,n$, are adapted
to the foliation, and $\partial_{x^a}=e_a$, $\partial_{y^a}=Je_a$,
where $z^a= x^a+iy^a$. By construction, the metric $g$ in the real
coordinates $\{w^1,w^2,z^1,z^2,x^a,y^a\}$ is written
\begin{multline}\label{metric real coord}
g  =  dw^1dz^1+dw^2dz^2+b(dw^1dw^1+dw^2dw^2)
+\sum_{a=1}^nr_a(dx^adw^1+dy^adw^2)\\+\sum_{a=1}^ns_a(dx^adw^2-dy^adw^1)
+\sum_{a=1}^n\epsilon_a(dx^adx^a+dy^ady^a),
\end{multline}
for some functions $b,r_a,s_a$ only depending on the variables $w^1$ and $w^2$, and $\epsilon_a=\pm 1$.
The inverse of this metric is
\begin{equation}\label{inverse metric}(g^{\mu\nu})=\begin{pmatrix}
0 & 0 & 1 & 0 & 0 & 0 & \ldots & 0 & 0 \\
0 & 0 & 0 & 1 & 0 & 0 & \ldots & 0 & 0 \\
1 & 0 & B & 0 & -\epsilon_1r_1 & \epsilon_1s_1 & \ldots & -\epsilon_nr_n & \epsilon_ns_n \\
0 & 1 & 0 & B & -\epsilon_1s_1 & -\epsilon_1r_1 & \ldots & -\epsilon_ns_n & -\epsilon_nr_n \\
0 & 0 & -\epsilon_1r_1 & -\epsilon_1s_1 & \epsilon_1 & 0 & & & \\
0 & 0 & \epsilon_1s_1 & -\epsilon_1r_1  & 0 & \epsilon_1 & & & \\
\vdots & \vdots & \vdots & \vdots & & & \ddots & & \\
0 & 0 & -\epsilon_nr_n & -\epsilon_ns_n & & & & \epsilon_n & 0  \\
0 & 0 & \epsilon_ns_n & -\epsilon_nr_n  & & & & 0 & \epsilon_n
\end{pmatrix}\end{equation}
where $B=-b+\sum_a\epsilon_a(r_a^2+s_a^2)$. Forcing this metric to
fulfill $\nabla \partial_{x^a}=\nabla \partial_{y^a}=0$,
$a=1,\ldots,n$, the following Christoffel symbols are zero
\begin{eqnarray*}
\Gamma^{z^2}_{w^1x^a}=\frac{1}{2}\left\{ \frac{\partial
s_a}{\partial w^1}-\frac{\partial r_a}{\partial w^2} \right\}, &
\qquad & \Gamma^{z^1}_{w^2x^a}=\frac{1}{2}\left\{ \frac{\partial
r_a}{\partial w^2}-\frac{\partial s_a}{\partial w^1} \right\},\\
\Gamma^{z^2}_{w^1y^a}=\frac{1}{2}\left\{ -\frac{\partial
s_a}{\partial w^2}-\frac{\partial r_a}{\partial w^1} \right\}, &
\qquad &  \Gamma^{z^1}_{w^2y^a}=\frac{1}{2}\left\{ \frac{\partial
r_a}{\partial w^1}+\frac{\partial s_a}{\partial w^2} \right\}.\\
\end{eqnarray*}
Hence the functions $r_a,s_a$ satisfy
\begin{equation}\label{Cauchy-Riemann}
\left.
\begin{array}{c}
\dfrac{\partial s_a}{\partial w^1}=\dfrac{\partial r_a}{\partial
w^2}\vspace{2mm}\\
\dfrac{\partial s_a}{\partial w^2}=-\dfrac{\partial r_a}{\partial
w^1}.
\end{array}
\right\}
\end{equation}
We can now write the Hermitian metric $g$ in the complex
coordinates $\{w,z,z^a\}$ as
\begin{eqnarray*}
h & = & \frac{1}{2}\left( dw\otimes d\overline{z}+dz\otimes
d\overline{w}\right)+\frac{b}{2}dw\otimes
d\overline{w}+\frac{1}{2}\sum_{a=1}^n\epsilon_a dz^a\otimes
d\overline{z^a}\\
& & +\sum_{a=1}^n\frac{h_a}{2}dw\otimes
d\overline{z^a}+\sum_{a=1}^n\frac{\overline{h_a}}{2}dz^a\otimes
d\overline{w},
\end{eqnarray*}
where $h_a=r_a+i(-s_a)$ is holomorphic (and $\overline{h_a}$ is
anti-holomorphic) for $a=1,\ldots,n$.

\bigskip

The remaining (possibly) non-zero Christoffel symbols are
\begin{eqnarray*}
\Gamma_{w^1w^2}^{z^1} & = & \tfrac{1}{2}\frac{\partial b}{\partial
w^2}-\tfrac{1}{2}\sum_a r_a\left( \frac{\partial r_a}{\partial
w^2}+\frac{\partial s_a}{\partial w^1} \right)-\tfrac{1}{2}\sum_a
s_a\left(-\frac{\partial s_a}{\partial
w^2}+\frac{\partial r_a}{\partial w^1}\right)\\
\Gamma_{w^1w^2}^{z^2} & = & \tfrac{1}{2}\frac{\partial b}{\partial
w^1}+\tfrac{1}{2}\sum_a s_a\left( \frac{\partial r_a}{\partial
w^2}+\frac{\partial s_a}{\partial w^1} \right)-\tfrac{1}{2}\sum_a
r_a\left(-\frac{\partial s_a}{\partial
w^2}+\frac{\partial r_a}{\partial w^1}\right)\\
\Gamma_{w^1w^1}^{z^1} & = & \frac{1}{2}\frac{\partial b}{\partial
w^1}+\sum_a(-r_a) \frac{\partial r_a}{\partial
w^1}+\sum_as_a\frac{\partial s_a}{\partial
w^1}\\
\Gamma_{w^1w^1}^{z^2} & = & -\frac{1}{2}\frac{\partial b}{\partial
w^2}+\sum_as_a\frac{\partial r_a}{\partial
w^1}+\sum_ar_a\frac{\partial s_a}{\partial
w^1}\\
\Gamma_{w^2w^2}^{z^1} & = & -\frac{1}{2}\frac{\partial b}{\partial
w^1}+\sum_a(-r_a)\frac{\partial s_a}{\partial
w^2}+\sum_a(-s_a)\frac{\partial r_a}{\partial
w^2}\\
\Gamma_{w^2w^2}^{z^2} & = & \frac{1}{2}\frac{\partial b}{\partial
w^2}+\sum_as_a\frac{\partial s_a}{\partial
w^2}+\sum_a(-r_a)\frac{\partial r_a}{\partial
w^2}\\
\Gamma^{x^a}_{w^1w^2} & = & \frac{1}{2}\left( \frac{\partial
r_a}{\partial w^2}+\frac{\partial s_a}{\partial w^1}\right)\\
\Gamma^{x^a}_{w^1w^1} & = & \frac{\partial r_a}{\partial w^1}\\
\Gamma^{x^a}_{w^2w^2} & = & \frac{\partial s_a}{\partial w^2}\\
\Gamma^{y^a}_{w^1w^2} & = & \frac{1}{2}\left( -\frac{\partial
s_a}{\partial w^2}+\frac{\partial r_a}{\partial w^1}\right)\\
\Gamma^{y^a}_{w^1w^1} & = & -\frac{\partial s_a}{\partial w^1}\\
\Gamma^{y^a}_{w^2w^2} & = & \frac{\partial r_a}{\partial w^2}.
\end{eqnarray*}
From them, we can thus compute
\begin{eqnarray*}
R(\partial_{w^1},\partial_{w^2},\partial_{w^1},\partial_{w^2}) & =
& \frac{1}{2}\frac{\partial^2 b}{\partial w^1 \partial w^1
}+\sum_a\left( \frac{1}{2}\frac{\partial s_a}{\partial
w^1}\frac{\partial r_a}{\partial
w^2}+\frac{1}{2}s_a\frac{\partial^2 r_a }{\partial w^1 \partial
w^2 }\right.\\
& + & \frac{1}{2}\left(\frac{\partial s_a}{\partial w^1}\right)^2
+\frac{1}{2}s_a\frac{\partial^2 s_a }{\partial w^1\partial w^1 }+\frac{1}{2}\frac{\partial r_a}{\partial w^1}\frac{\partial s_a}{\partial w^2}\\
& + & \left. \frac{1}{2}r_a\frac{\partial^2 s_a }{\partial
w^1\partial w^2 }-\frac{1}{2}\left(\frac{\partial r_a}{\partial
w^1}\right)^2-\frac{1}{2}r_a\frac{\partial^2 r_a}{\partial w^1
\partial w^1}\right)\\
& + & \frac{1}{2}\frac{\partial^2 b}{\partial w^2\partial
w^2}+\sum_a\left( -\frac{\partial s_a}{\partial w^2}\frac{\partial
r_a}{\partial w^1}-s_a\frac{\partial^2 r_a}{\partial w^1\partial
w^2} \right.\\
& - & \left. \frac{\partial r_a}{\partial w^2}\frac{\partial
s_a}{\partial w^1}-r_a\frac{\partial^2 r_a}{\partial w^1\partial
w^2}\right)\\
& + & \sum_as_a\left(\tfrac{1}{2}\frac{\partial^2 r_a}{\partial
w^1\partial w^2 }+ \tfrac{1}{2}\frac{\partial^2 s_a}{\partial
w^1\partial w^1}-\frac{\partial^2 r_a}{\partial w^1\partial w^2
}\right)\\
& + & \sum_ar_a\left(-\tfrac{1}{2}\frac{\partial^2 s_a}{\partial
w^1\partial w^2 }+\tfrac{1}{2}\frac{\partial^2 r_a}{\partial
w^1\partial w^1 }+\frac{\partial^2 s_a}{\partial w^1\partial
w^2 }\right).
\end{eqnarray*}
Using (\ref{Cauchy-Riemann}) we obtain
$$
R(\partial_{w^1},\partial_{w^2},\partial_{w^1},\partial_{w^2})=\frac{1}{2}\Delta
b,
$$
where $\Delta$ is the Laplace operator with respect to the variables $(w^1,w^2)$, so the curvature is
\begin{equation}
\label{11}
R=\frac{1}{2}\Delta b(dw^1\wedge dw^2)\otimes(dw^1\wedge
dw^2).
\end{equation}
One can see that all Christoffel symbols of the form
$\Gamma^{w^k}_{\cdot\cdot}$, $k=1,2$, vanish, whence the covariant
derivative of $R$ is
\begin{eqnarray}
\label{22}
 \nabla R & = & \frac{1}{2}\frac{\partial}{\partial
w^1}(\Delta
b)dw^1\otimes(dw^1\wedge dw^2)\otimes(dw^1\wedge dw^2)\\
\nonumber & + & \frac{1}{2}\frac{\partial}{\partial w^2}(\Delta
b)dw^2\otimes(dw^1\wedge dw^2)\otimes(dw^1\wedge dw^2).
\end{eqnarray}
On the other hand, solving for $\theta$ in (\ref{equatio dw1 dw2})
gives
\begin{equation}
\label{33} \theta=-\frac{1}{(w^1)^2+(w^2)^2}(w^1dw^1+w^2dw^2).
\end{equation}
Taking into account the relation \eqref{11}, \eqref{22} and
\eqref{33} in the equation $\nabla R=4\theta\otimes R$ given in
Lemma \ref{Lemma BGO}, we have
the system of partial differential equations
\begin{eqnarray*}
\frac{\partial}{\partial w^1}(\Delta b) & = &
\frac{-4w^1}{(w^1)^2+(w^2)^2}\Delta b,\\
\frac{\partial}{\partial w^2}(\Delta b) & = &
\frac{-4w^2}{(w^1)^2+(w^2)^2}\Delta b,
\end{eqnarray*}
which can be integrated to give the Poisson equation
$$\Delta b=\frac{b_0}{\left((w^1)^2+(w^2)^2\right)^2},$$
for some constant $b_0\in\R.$
\end{proof}

\bigskip

\begin{corollary}
The curvature of the metric $g$ with respect to the coordinates
$\{w^1,w^2,z^1,z^2,x^a,y^a\}$ is
\begin{equation}\label{curvatura (0,4)}
R=\frac{1}{2}\frac{b_0}{\left((w^1)^2+(w^2)^2\right)^2}(dw^1\wedge
dw^2)\otimes(dw^1\wedge dw^2)\qquad b_0\in\bb{R},
\end{equation}
hence $(M,g,J)$ is a pseudo-K\"ahler Ricci-flat manifold.
\end{corollary}
Note that $(M,g,J)$ is flat if and only if $b_0=0$.

\begin{corollary}
In complex coordinates $w=w^1+iw^2$, $z=z^1+iz^2$, $z^a=x^a+iy^a$,
$a=1,\ldots,n$, the metric \eqref{metric2} is expressed
\begin{eqnarray}\label{metrica comp 1}
h & = & \frac{1}{2}\left( dw\otimes d\overline{z}+dz\otimes
d\overline{w}\right)+\frac{b}{2}dw\otimes
d\overline{w}+\sum_{a=1}^n\frac{h_a}{2}dw\otimes
d\overline{z^a}\nonumber\\
& & +\sum_{a=1}^n\frac{\overline{h_a}}{2}dz^a\otimes
d\overline{w}+\frac{1}{2}\sum_{a=1}^n\epsilon_a dz^a\otimes
d\overline{z^a},
\end{eqnarray}
where $\epsilon_a=\pm 1$, $h_a:\bb{C}\to\bb{C}$ is a holomorphic
function of the variable $w$ for all $a=1,\ldots,n$, and
$b:\bb{C}\to\bb{R}$ is a function of $w$ satisfying the Poisson
equation
$$\Delta b=\frac{b_0}{||w||^4}$$
for some $b_0\in\bb{R}$.
\end{corollary}

\subsection{Some global properties}

In this section we derive some special properties of a pseudo-K\"ahler manifold $(M,g,J)$ admitting a strongly degenerate homogeneous pseudo-K\"ahler structure of linear type:

\begin{proposition}\label{prop hol}
If $b_0\neq 0$, the holonomy algebra of $(M,g,J)$ can be
identified with the one dimensional space
$$\f{hol}=\bb{R}\left(\begin{array}{ccl}i & i & 0\\ -i & -i & 0\\ 0 & 0 & 0_n\end{array}\right)\subset\f{su}(1,1)\subset\f{su}(p,q).$$
\end{proposition}

\begin{proof}
The fact that the manifold is Ricci-flat implies that the
restricted holonomy is contained in $SU(p,q)$ if the signature of
$(M,g)$ is $(2p,2q)$. Moreover, in the proof of Theorem \ref{prop
princ} it is shown that the holonomy group of the neighborhood
$(U,g)$ at $p$ has two invariant subspaces:
$E=\mathrm{span}\{\partial_{w^1},\partial _{w^2},\partial
_{z^1},\partial _{z^1}\}$ and $E^{\bot}$. This means that the
local holonomy representation is decomposable into the trivial
representation on $E^{\bot}$ and a representation of a group $H
\subset SU(1,1)$ on $E$. On the other hand, since $(M,g)$ is a
complex manifold, it is in particular real analytic, so the
restricted holonomy, the local holonomy and the infinitesimal
holonomy coincides (see \cite[Vol.I, Ch. II]{KN}). Recall that the
infinitesimal holonomy algebra at $p\in M$ is defined as
$\f{hol}'=\bigcup_{k=0}^{\infty}\f{m}_k$, where
$$\f{m}_0=\mathrm{span}\{R_{XY}/X,Y\in T_pM\}$$
and
$$\f{m}_k=\mathrm{span}\left\{\f{m}_{k-1}\cup\{(\nabla_{Z_k}\ldots\nabla_{Z_1}R)_{XY}/Z_1,\ldots,Z_k,X,Y\in T_pM\}\right\}.$$
In our case, by the recurrent
formula $\nabla R=4\theta\otimes R$ we have
$\f{m}_k=\f{m}_{k-1}=\ldots=\f{m}_1=\f{m}_0$, hence
$\f{hol}'=\f{m}_0$.
In addition, by (\ref{curvature
op}) $\f{m}_0$ is the one dimensional space
$$\f{m}_0=\mathrm{span}\{R_{\partial_{w^1}\partial_{w^2}}\}=\mathrm{span}\{A\}$$
where
\begin{equation}\begin{array}{rrcl}
A: & T_pM & \to & T_pM \label{holA}\\
                                   & \partial_{w^1} & \mapsto &
                                   \partial_{z^2} \\
                                   & \partial_{w^2} & \mapsto &
                                   -\partial_{z^1}\\
                                   & \partial_{z^1},\partial_{z^2} & \mapsto & 0\\
                                   & \partial_{x^a},\partial_{y^a} & \mapsto & 0.
\end{array}\end{equation}
Let $s$ be the sign of $b$ at $p$. With respect to the unitary
basis $\{W,Z\}$ of $E$ defined as
$$W=\frac{\partial_{w^1}}{\sqrt{|b|}}-i\frac{\partial_{w^2}}{\sqrt{|b|}},
\hspace{1em}Z=\left(\frac{\partial_{w^1}}{\sqrt{|b|}}-s\sqrt{|b|}\partial_{z^1}\right)
-i\left(\frac{\partial_{w^2}}{\sqrt{|b|}}-s\sqrt{|b|}\partial_{z^2}\right),$$
the matrix of the endomorphism $A$ restricted to $E$ is
$$A_{|E}=s\begin{pmatrix}i & i\\ -i & -i\end{pmatrix}\in\f{su}(1,1).$$
\hspace{75mm}
\end{proof}

\begin{proposition}
$(M,g,J)$ is a Walker
manifold.
\end{proposition}

\begin{proof}
 Recall that the distribution of $2$-planes spanned by
$\partial_{z^1}$ and $\partial_{z^2}$ is invariant by holonomy.
Since $\xi$ and $J\xi$ are linear combination of $\partial_{z^1}$
and $\partial_{z^2}$, these vector fields generates a null
parallel distribution. Therefore, $(M,g,J)$ is a Walker
manifold (see \cite{BGGNV}).
\end{proof}

\bigskip

A \textit{scalar invariant} is a scalar function obtained by fully contracting the curvature $R$ and its derivatives $\nabla^k R$ with the metric and the inverse of the metric.

\begin{proposition}\label{proposition VSI}
$(M,g,J)$ is VSI (vanishing scalar invariants).
\end{proposition}

\begin{proof}
Since the curvature (\ref{curvatura (0,4)}) only involves $dw^1$ and $dw^2$, it is easy to see that the  inverse metric (\ref{inverse metric}) forces all possible scalar invariants to vanish.
\end{proof}

\begin{proposition}
$(M,g,J)$ is an \textit{Osserman manifold} with a $2$-step nilpotent Jacobi operator.
\end{proposition}

\begin{proof}
The Jacobi operators, that is $\mathrm{J}(X): Y\mapsto R_{YX}X$, of the elements of the basis $\{\partial_{w^1},\partial_{w^2},\partial_{z^1},\partial_{z^2},\partial_{x^a},\partial_{y^a}\}$ are
$$\mathrm{J}(\partial_{w^1})=\left(\begin{array}{ccccl}0 & 0 & 0 & 0 & 0\\ 0 & 0 & 0 & 0 & 0\\ 0 & 0 & 0 & 0 & 0\\
0 & -\frac{b_0}{2||w||^4} & 0 & 0 & 0\\ 0 & 0 & 0 & 0 & 0_{2n}\end{array}\right),$$
$$\mathrm{J}(\partial_{w^2})\left(\begin{array}{ccccl}0 & 0 & 0 & 0 & 0\\ 0 & 0 & 0 & 0 & 0\\ -\frac{b_0}{2||w||^4} & 0 & 0 & 0 & 0\\
0 & 0 & 0 & 0 & 0\\ 0 & 0 & 0 & 0 & 0_{2n}\end{array}\right),$$
$$\mathrm{J}(\partial_{z^1})=0,\qquad \mathrm{J}(\partial_{z^2})=0,\qquad \mathrm{J}(\partial_{x^a})=0,\qquad \mathrm{J}(\partial_{y^a})=0.$$
\hspace{100mm}
\end{proof}

\bigskip

\subsection{The manifold $(\bb{C}^{2+n},h)$}

Theorem \ref{prop princ} gives the local expression
(\ref{metric1}) of the metric of a manifold with a strongly
degenerate homogeneous K\"ahler structure of linear type. This
motivates the study of the space $\bb{C}^{2+n}$ endowed with this
particular pseudo-K\"ahler metric, which can thus be understood as
the simplest instance of this type of manifolds. In particular,
the goal of this section is to study the singular nature of this
space, and to explicitly exhibit the corresponding strongly
degenerate homogeneous pseudo-K\"ahler structures of linear type.
In addition, this results will be applied in \S \ref{secwave}.

We thus consider $(\bb{C}^{2+n},h)$ with the standard complex
structure of $\bb{C}^{2+n}$ and the pseudo-K\"ahler metric $h$
given in \eqref{metrica comp 1}. With respect to the standard real
coordinates $\{w^1,w^2,z^1,z^2,x^a,y^a\}$, $a=1,\ldots,n$, of
$\bb{C}^{2+n}\equiv\bb{R}^{4+2n}$, the metric has the form
\eqref{metric real coord}
where the functions $b,r_a,s_a$ only depend on the variables $w^1$
and $w^2$ and satisfy
\begin{equation}\label{3 equations}
\left.
\begin{array}{l}
\dfrac{\partial s_a}{\partial w^1}=\dfrac{\partial r_a}{\partial
w^2}\vspace{2mm}\\
\dfrac{\partial s_a}{\partial w^2}=-\dfrac{\partial r_a}{\partial
w^1}\vspace{2mm}\\
\Delta b=\dfrac{b_0}{\left((w^1)^2+(w^2)^2\right)^2},\hspace{1em}
b_0\in\bb{R}.
\end{array}\right\}
\end{equation}
Note, that the $(0,3)$-Riemann curvature tensor is
\begin{eqnarray*}
R
&=&\frac{1}{2}\frac{b_0}{\left((w^1)^2+(w^2)^2\right)^2}\left((dw^1\wedge
dw^2)\otimes(dw^1\otimes \partial_{z^2})\right.\\
& & \hspace{27mm}\left.-(dw^1\wedge dw^2)\otimes(dw^2\otimes
\partial_{z^1})\right),
\end{eqnarray*}
so if $b_0\neq 0$ it exhibits a singular behavior at
$(w^1,w^2)=(0,0)$. The set $\{w^1=w^2=0\}$ can be understood as a
singularity of $g$ in the cosmological sense:
\begin{enumerate}
\item The geodesic deviation equation is governed by the
components of the curvature tensor $R_{w^1w^2w^i}^{z^j}$,
$i,j=1,2$, making the tidal forces infinite at $\{w^1=w^2=0\}$.

\item The geodesic equation for the variables $(w^1,w^2)$ are
\[
\ddot{w}^1=0,\qquad \ddot{w}^2=0.
\]
Then, geodesics with initial value
$$(w^1(0),w^2(0))=(a,b),\qquad (\dot{w}^1(0),\dot{w}^2(0))=(c,cb/a),$$ with $b\in\bb{R}$,
$a,c\in\bb{R}-\{0\}$, are
$$w^1(t)=ct+a,\qquad w^2(t)=c\frac{b}{a}t+b,$$
and geodesics with initial value $$(w^1(0),w^2(0))=(0,b), \qquad
(\dot{w}^1(0),\dot{w}^2(0))=(0,d),$$ with $b,d\in\bb{R}-\{0\}$,
are
$$w^1(t)=0,\qquad w^2(t)=dt+b.$$ These geodesics reach the singular
set $\{w^1=w^2=0\}$ in finite time $t=-\frac{a}{c}$ and
$t=-\frac{b}{d}$ respectively. Hence
$(\bb{C}^{2+n}-\{w^1=w^2=0\},h)$ is not geodesically complete.
\end{enumerate}

\bigskip

\bigskip

There exists the possibility that this singularity is due to a bad
choice of coordinates so that we could embed
$(\C^{2+n}-\{w^1=w^2=0\},h)$ in a complete manifold
$(\hat{M},\hat{h})$ with $\hat{h}$ of class $\mathcal{C}^2$ . To
see that this is actually not possible, we can compute a component
of the curvature tensor with respect to an orthonormal parallel
frame along a curve reaching the singular set in finite time, and
see that it is singular (see \cite{Sen}). Indeed, let $\gamma$ be
the geodesic with initial value $\gamma(0)=(1,0,\ldots,0)$ and
$\dot{\gamma}=(-1,0,\ldots,0)$. We have seen that this geodesic is
of the form
$$\gamma(t)=(1-t,0,z^1(t),z^2(t),x^a(t),y^a(t))$$
for some functions $z^1(t),z^2(t),x^a(t),y^a(t)$, $a=1,\ldots,n$.
Let
$$E(t)=W^1(t)\partial_{w^1}+W^2(t)\partial_{w^2}+Z^1(t)\partial_{z^1}+Z^2(t)\partial_{z^2}+X^a(t)\partial_{x^a}+Y^a(t)\partial_{y^a}$$
be a vector field along $\gamma$. $E$ is parallel, i.e.
$\nabla_{\dot{\gamma}}E=0$, if the following equations hold:
\begin{align*}
0=\dot{W}^1,& \qquad 0=\dot{W}^2,\\
0=\dot{Z}^1-W^1\Gamma_{w^1w^1}^{z^1}-W^2\Gamma_{w^1w^2}^{z^1},&
\qquad
0=\dot{Z}^2-W^1\Gamma_{w^1w^1}^{z^2}-W^2\Gamma_{w^1w^2}^{z^2},\\
0=\dot{X}^a-W^1\Gamma_{w^1w^1}^{x^a}-W^2\Gamma_{w^1w^2}^{x^a},&
\qquad
0=\dot{Y}^a-W^1\Gamma_{w^1w^1}^{y^a}-W^2\Gamma_{w^1w^2}^{y^a}.
\end{align*}
We can thus obtain an orthonormal parallel frame
$\{E_1(t),\ldots,E_{4+2n}(t)\}$ with $E_1(t)$ and $E_2(t)$ of the
form
\begin{align*}
E_1(t)=& \frac{1}{\sqrt{|b(0)|}}\partial_{w^1}+Z_1^1(t)\partial_{z^1}+Z_1^2(t)\partial_{z^2}+X_1^a(t)\partial_{x^a}+Y_1^a\partial_{y^a},\\
E_2(t)=&
\frac{1}{\sqrt{|b(0)|}}\partial_{w^2}+Z_2^1(t)\partial_{z^1}+Z_2^2(t)\partial_{z^2}+X_2^a(t)\partial_{x^a}+Y_2^a\partial_{y^a},
\end{align*}
where $E_1(0)=\frac{1}{\sqrt{|b(0)|}}\partial_{w^1}$,
$E_2(0)=\frac{1}{\sqrt{|b(0)|}}\partial_{w^2}$, and
$b(0)=b(1,0,\ldots,0)$. The value of the curvature tensor applied
to $E_1(t),E_2(t)$ is
$$R_{E_1(t)E_2(t)E_1(t)E_2(t)}=\frac{b_0}{2b(0)^2}\frac{1}{(w^1(t)^2+w^2(t)^2)^2}=\frac{b_0}{2b(0)^2}\frac{1}{(1-t)^4},$$
which is singular at $t=1$, that is when $\gamma$ reaches the
singular set $\{w^1=w^2=0\}$.

\bigskip

Finally we show that strongly degenerate homogeneous
pseudo-K\"ahler structures of linear type indeed exist and are
realized in the manifold $(\bb{C}^{2+n}-\{w^1=w^2=0\},h)$.
\begin{proposition}
For every data $(b,b_0,r_a,s_a)$, $a=1,\ldots,n$, satisfying
\eqref{3 equations}, the pseudo-K\"ahler manifold
$(\bb{C}^{2+n}-\{w^1=w^2=0\},h)$ admits a strongly degenerate
pseudo-K\"ahler homogeneous structure of linear type.
\end{proposition}

\begin{proof}
Let $S$ be the tensor field
$$S_XY=g(X,Y)\xi-g(Y,\xi)X-g(X,JY)J\xi+g(JY,\xi)JX,$$
with
$$\xi=\frac{-1}{(w^1)^2+(w^2)^2}(w^1\partial_{z^1}+w^2\partial_{z^2}).$$
It is a straightforward computation to see that $\wnabla \xi=0$
and $\wnabla R=0$, where $\wnabla=\nabla-S$, so that $S$ satisfies
equations \eqref{AS+J}.
\end{proof}

\bigskip


\section{The homogeneous model for a strongly degenerate homogeneous structure of linear
type}\label{section homogeneous model}

Let $(M,g,J)$ be a reductive homogeneous pseudo-K\"ahler manifold
admitting a strongly degenerate homogeneous structure of linear
type $S$. From \cite{TV} one can construct a Lie algebra of
infitesimal isometries associated to $S$. This algebra is (fixing
a point $p\in M$ as the origin)
$$\f{g}=T_pM\oplus \f{hol}^{\widetilde{\nabla}}$$
where $\widetilde{\nabla}=\nabla-S$ is the canonical connection
associated to the homogeneous structure tensor $S$. The brackets
in $\f{g}$ are
$$\left\{\begin{array}{lcll}
\left[A,B\right] & = & AB-BA, & A,B\in \f{hol}^{\widetilde{\nabla}}\\
\left[A,\eta\right] & = & A\cdot\eta, &
A\in \f{hol}^{\widetilde{\nabla}},\eta\in T_pM\\
\left[\eta,\zeta\right] & = &
S_{\eta}\zeta-S_{\zeta}\eta+\widetilde{R}_{\eta\zeta}, &
\eta,\zeta\in T_pM,
\end{array}\right.$$
 where $\widetilde{R}$ is the
curvature tensor of $\widetilde{\nabla}$. This curvature tensor
can be computed as $\widetilde{R}=R-R^S$ where
$$R^S_{XY}Z=\left[S_X,S_Y\right]Z-S_{S_XY-S_YX}Z.$$
Let $G$ be a Lie group with Lie algebra $\f{g}$, and let $H$ be
the connected Lie subgroup with Lie algebra
$\f{hol}^{\widetilde{\nabla}}$. If $H$ is closed in $G$, then
$G/H$ is called a \textit{homogeneous model} for $M$, which means that $(M,g,J)$
is locally holomorphically isometric to $G/H$ with the $G$-invariant
metric and complex structure given by $g$ and $J$ at $T_pM$.

By direct calculation from (\ref{degenerate}) one finds
$$R^S_{XY}Z=-2g(X,JY)\left(g(\xi,JZ)\xi+g(Z,\xi)J\xi\right),$$
which is readily seen to equal
$$\left(R^S_p\right)_{XY}=-2g(X,JY)A,$$
where $A$ is the endomorphism given in (\ref{holA}). In terms of
the coordinate system $\{w^1,w^2,z^1,z^2,x^a,y^a\}$ around $p$
described in Theorem \ref{prop princ} (with $p$ sent to
$(1,0,\ldots,0)\in\bb{C}^{n+2}$), and taking into account
(\ref{curvature op}), we can write
$$\widetilde{R}_p=\left(\frac{b_0}{2}dw^1\wedge dw^2+2\omega_p\right)\otimes A,$$
where $\omega$ is the K\"ahler form associated to $(g,J)$. This
means that $\f{hol}^{\widetilde{\nabla}}=\f{hol}\simeq
\f{u}(1)\subset \f{su}(1,1)$.

The (real) Lie algebra $\f{g}$ is then
$$\f{g}\simeq \f{u}(1)\oplus\bb{C}^2\oplus\bb{C}^n,$$ where
$\bb{C}^2=\bb{C}\{\partial_{w^1},\partial_{z^1}\}$ and
$\bb{C}^n=\bb{C}\{\partial_{x^a},a=1,\ldots,n\}$. Setting
$B=\partial_{z^2}-A$, the brackets are
\begin{equation}\label{Lie algebra g}
\begin{array}{l}
\left[A_1,A_2\right]=0,\\
\left[B,\partial_{z^1}\right]=0,\hspace{1em}\left[B,\partial_{z^2}\right]=0,\hspace{1em}\left[B,\partial_{x^a}\right]=0,\hspace{1em}\left[B,\partial_{y^a}\right]=0,\\
\left[B,\partial_{w^1}\right]=2B,\hspace{1em}\left[B,\partial_{w^2}\right]=0,\\
\left[\partial_{z^1},\partial_{z^2}\right]=0,\hspace{1em}\left[\partial_{w^1},\partial_{w^2}\right]=\left(-2b(p)+\frac{1}{2}b_0\right)B-\frac{1}{2}b_0\partial_{z^2}\\
\left[\partial_{z^1},\partial_{w^1}\right]=\partial_{z^1},\hspace{1em}\left[\partial_{z^1},\partial_{w^2}\right]=\partial_{z^2}-2B,\\
\left[\partial_{z^2},\partial_{w^1}\right]=\partial_{z^2}+2B,\hspace{1em}\left[\partial_{z^2},\partial_{w^2}\right]=-\partial_{z^1},\\
\left[\partial_{x^a},\partial_{x^b}\right]=\left[\partial_{y^a},\partial_{y^b}\right]=0,\\
\left[\partial_{x^a},\partial_{y^b}\right]=-2\delta^a_bB,\\
\left[\partial_{z^k},\partial_{x^a}\right]=\left[\partial_{z^k},\partial_{y^a}\right]=0,\hspace{1em}
k=1,2,\\
\left[\partial_{w^1},\partial_{x^a}\right]=-\partial_{x^a},\hspace{1em}\left[\partial_{w^1},\partial_{y^a}\right]=-\partial_{y^a},\\
\left[\partial_{w^2},\partial_{x^a}\right]=-\partial_{y^a},\hspace{1em}\left[\partial_{w^2},\partial_{y^a}\right]=\partial_{x^a},
\end{array}\end{equation}
for $a,b=1,\ldots,n$ and $A_1,A_2\in\f{u}(1)$.

One can check that $\f{g}$ is a solvable Lie algebra with a
$2$-step nilradical
$\f{n}=\bb{R}\left\{A,\partial_{z^1},\partial_{z^2},\partial_{x^a},\partial_{y^a},a=1,\ldots,n\right\}$.
Note also that $\f{g}$ contains a Heisenberg algebra
$\f{h}=\bb{R}\left\{B,\partial_{x^a},\partial_{y^a},a=1,\ldots,n\right\}$.

Using Lie's Theorem (\cite{Ful}) we can obtain an upper triangular
representation of the Lie algebra $\f{g}$ as shown in Table
\ref{tabla}, where $t,w_1,w_2,z_1,z_2$,
$x_1,y_1,\ldots,x_n,y_n\in\bb{R}$, $\lambda=2b(p)$ and
$\mu=\frac{b_0}{2}$.
\begin{sidewaystable}
\centering $\begin{pmatrix}

-2w_1 & 2w_2+2w_1i & 2w_2+2w_1i & -4y_1i & -4x_1i & \ldots &
-4y_ni & -4x_ni &
2t+(\lambda-\mu)w_2+2(z_1-z_2)i & -4+(\mu-\lambda)w_1\\

0 & -w_1+w_2i &  0 & 0 & 0 & \ldots & 0 & 0 &
z_1-\frac{\mu}{2}w_2i
& \frac{\mu}{2}w_1i\\

0 & 0 & -w1-w2i & 0 & 0 & \ldots & 0 & 0 & z_2-\frac{\mu}{2}w_2i &
z_1+z_2i-\frac{\mu}{2}w_1i\\

0 & 0 & 0 & -w_1+w_2i & 0 & \ldots & 0 & 0 & x_1 & -x1i\\

0 & 0 & 0 & 0 & -w_1+w_2i & \ldots & 0 & 0 & y_1 & y_1i\\

\vdots  & \vdots & \vdots & \vdots & \vdots & \ddots & \vdots & \vdots & \vdots & \vdots \\

0 & 0 & 0 & 0 & 0 & \ldots & -w_1+w_2i & 0 & x_n & -x_ni\\

0 & 0 & 0 & 0 & 0 & \ldots & 0 & -w_1+w_2i & y_n & y_ni\\

0 & 0 & 0 & 0 & 0 & \ldots & 0 & 0 & 0 & 0\\

0 & 0 & 0 & 0 & 0 & \ldots & 0 & 0 & 0 & 0
\end{pmatrix}$
      \caption{Lie algebra $\f{g}$ of $G$}\label{tabla}
\end{sidewaystable}

The isotropy algebra is
generated by the matrix
$$
\begin{pmatrix}
0 & 0 & 0 & 0 & 0 & \ldots & 0 & 0 & 2 & 0\\
0 & 0 & 0 & 0 & 0 & \ldots & 0 & 0 & 0 & 0\\
\vdots  & \vdots & \vdots & \vdots & \vdots & \ddots & \vdots & \vdots & \vdots & \vdots \\
0 & 0 & 0 & 0 & 0 & \ldots & 0 & 0 & 0 & 0
\end{pmatrix},
$$
so exponentiating we obtain a Lie group $G$ and a closed subgroup
$H$ defining a homogeneous model for $M$. The metric $g$ and the
complex structure $J$ give a bilinear form and a complex structure
on
$$\f{m}=\mathrm{span}\{\partial_{w^1},\partial_{w^2},\partial_{z^1},\partial_{z^2},\partial_{x^a},\partial_{y^a}\}$$
respectively. These induce a $G$-invariant metric $\bar{g}$ and a
$G$-invariant complex structure $\bar{J}$ on $G/H$ so that
$(M,g,J)$ is locally holomorphically isometric to
$(G/H,\bar{g},\bar{J})$.

\begin{proposition}
The homogeneous model $(G/H,\bar{g},\bar{J})$ constructed above is not geodesically complete.
\end{proposition}

\begin{proof}
Let $\sigma$ be the Lie
algebra involution of $\f{g}$ given by
$$\begin{array}{rrcl}
\sigma: & \f{g} & \to & \f{g}\\
        & B & \mapsto & -B\\
        & \partial_{w^1} & \mapsto & \partial_{w^1}\\
& \partial_{w^2} & \mapsto & -\partial_{w^2}\\
         & \partial_{z^1} & \mapsto & \partial_{z^1}\\
         & \partial_{z^2} & \mapsto & -\partial_{z^2}\\
         & \partial_{x^a} & \mapsto & \partial_{x^a}\\
& \partial_{y^a} & \mapsto & -\partial_{y^a}.
\end{array}$$
One can check that the restriction of $\sigma$ to $\f{m}$ is an
isometry with respect to the bilinear form given by $\bar{g}$. The
subalgebra of fixed points is
$\f{g}^{\sigma}=\bb{R}\left\{\partial_{w^1},\partial_{z^1},\partial_{x^a}\right\}$.
Working with the universal cover  if necessary we can assume that
$G$ is simply-connected so that $\sigma$ induces an involution in
$G$ and therefore an isometric involution in $G/H$. We will denote
all this involutions by $\sigma$. Let $G^{\sigma}$ be the
connected Lie subgroup of $G$ with Lie algebra $\f{g}^{\sigma}$,
note that  $\f{g}^{\sigma}\cap\f{hol}^{\widetilde{\nabla}}=\{0\}$,
so $\left(G/H\right)^{\sigma}=G^{\sigma}$, where
$\left(G/H\right)^{\sigma}$ stands for the fixed point set of
$\sigma:G/H\to G/H$. It is a well-known result that
$\left(G/H\right)^{\sigma}=G^{\sigma}$ is a closed totally
geodesic submanifold of $G/H$. Let now $\theta$ be the Lie algebra
involution of $\f{g}^{\sigma}$ given by
$$\begin{array}{rrcl}
\theta: & \f{g}^{\sigma} & \to & \f{g}^{\sigma}\\
        & \partial_{w^1} & \mapsto & \partial_{w^1}\\
         & \partial_{z^1} & \mapsto & \partial_{z^1}\\
         & \partial_{x^a} & \mapsto & -\partial_{x^a},
\end{array}$$
which is again an isometry with respect to the bilinear form
induced in $\f{g}^{\sigma}$ by restriction from $\f{m}$. The
subalgebra of fixed points is
$\f{k}=\left(\f{g}^{\sigma}\right)^{\theta}=\bb{R}\left\{\partial_{w^1},\partial_{z^1}\right\}$.
Let $\widetilde{G^{\sigma}}$ be the universal cover of
$G^{\sigma}$, $\theta:\f{g}^{\sigma}  \to  \f{g}^{\sigma}$ induces
an isometric involution
$\theta:\widetilde{G^{\sigma}}\to\widetilde{G^{\sigma}}$.
Therefore, let $K$ be the connected Lie subgroup of
$\widetilde{G^{\sigma}}$ with lie algebra $\f{k}$, $K$ is a
totally geodesic submanifold of $\widetilde{G^{\sigma}}$.

Let $s$ be the sign of $b(p)$. We define the left-invariant vector
fields in $\f{k}$
$$U=\frac{1}{\sqrt{|b(p)|}}\partial_{w^1},\hspace{2em} V=U-s\sqrt{|b(p)|}\partial_{z^1}.$$
We have
$$
<U,U>=s,\hspace{2em}<V,V>=-s,\hspace{2em}<U,V>=0,\\
$$
$$[U,V]=\frac{1}{\sqrt{|b(p)|}}(u-v),$$
 where $<\hspace{1mm},\hspace{1mm}>$ stands for the bilinear form inherited by $\f{k}$ from $\f{g}^{\sigma}$
and which determines the pseudo-Riemannian
left-invariant metric in $K$. The Levi-Civita connection of this
metric is
$$\begin{array}{cc}
\nabla_UU=\frac{1}{\sqrt{|b(p)|}}V, &
\nabla_UV=\frac{1}{\sqrt{|b(p)|}}U,\\
\nabla_VV=\frac{1}{\sqrt{|b(p)|}}U, &
\nabla_VU=\frac{1}{\sqrt{|b(p)|}}V.
\end{array}$$
Let $\gamma$ be a curve in $K$ and $\dot{\gamma}$ its tangent
vector. Setting $\dot{\gamma}(t)=u(t)U+v(t)V$, the geodesic
equation $\nabla_{\dot{\gamma}}\dot{\gamma}=0$ implies
$$
\begin{array}{l}
\dot{u}+\frac{1}{\sqrt{|b(p)|}}(uv+v^2)=0\\
\dot{v}+\frac{1}{\sqrt{|b(p)|}}(uv+u^2)=0.
\end{array}
$$
Changing variables to $x=u+v$ and $y=u-v$ the equations transform
into
$$
\begin{array}{l}
\dot{x}+\frac{1}{\sqrt{|b(p)|}}x^2=0\\
\dot{y}-\frac{1}{\sqrt{|b(p)|}}xy=0,
\end{array}
$$
the solutions of which are
$$x=\sqrt{|b(p)|}\frac{1}{t-c},\hspace{2em}y=A\exp\left(\sqrt{|b(p)|}\frac{1}{t-c}\right)$$
for some constants $A,c\in\bb{R}$. Therefore, $K$ is not
geodesically complete. Hence, since we have the following inclusions of totally geodesic submanifolds
$$K\subset\widetilde{G^{\sigma}},\qquad G^{\sigma}=\left(G/H\right)^{\sigma}\subset G/H,$$
the manifold $(G/H,g,J)$ is not geodesically complete.
\end{proof}

\begin{corollary}
Let $(M,g,J)$ be a connected and simply-connected pseudo-K\"ahler manifold
admitting a strongly degenerate pseudo-K\"ahler structure of linear
type $S$, then it is geodesically uncomplete.
\end{corollary}

\begin{proof}
Suppose that $(M,g,J)$ is geodesically complete. Ambrose-Singer
theorem assures that $(M,g,J)$ is (globally) holomorphically
isometric to the homogeneous model $(G/H,\bar{g},\bar{J})$. But
this homogeneous model is not geodesically complete.
\end{proof}

\bigskip

\section{Homogeneous structures and homogeneous plane
waves}\label{secwave}

%

\subsection{The Lorentz case}
\label{LKc}

\begin{definition}
A plane wave is the Lorentz manifold $M=\R^{n+2}$ with metric
$$g=dudv+A_{ab}(u)x^ax^bdu^2+\sum_{a=1}^n(dx^a)^2,$$
where $(A_{ab})$ is a symmetric matrix called the profile.
\end{definition}

A plane wave is called \textit{homogeneous} if the Lie algebra of
Killing vector fields acts transitively in the tangent space at
every point. It is well known \cite{Blau} that any plane wave
admits the following algebra of Killing vector fields
$$\mathrm{span}\{\partial_v,X_{p_a},X_{q_a};a=1,\ldots,n\},$$ where $p_a,q_a$  are
solutions of the harmonic oscillator equation $d^2f/du^2=A(u)f$
with initial values
$$\begin{matrix}
(p_a)_b(u_0)=\delta_{ab},&(\dot{p}_a)_b(u_0)=0\\
(q_a)_b(u_0)=0,&(\dot{q}_a)_b(u_0)=\delta_{ab},
\end{matrix}$$
and
$$X_f=f_a\partial_{x^a}-(df/du)_a x^a\partial_v.$$ In
particular, this algebra is isomorphic to the Heisenberg algebra.
Therefore, the homogeneity of a plane wave is characterized by the
existence of an extra Killing vector field with non zero
$u$-component. Homogeneous plane waves were classified in
\cite{Blau}. Among them there are two special types in which we
are interested: Cahen-Wallach spaces and singular scale-invariant
homogeneous plane waves.

The Cahen-Wallach space $M^{1,n+1}_{\lambda_1,\ldots,\lambda_n}$
is defined as a plane wave with metric
$$g=dudv+A_{ab}x^ax^bdu^2+\sum_{a=1}^n(dx^a)^2,$$ where $(A_{ab})$ is
a constant symmetric matrix with eigenvalues
$(\lambda_1,\ldots,\lambda_n)$. They are one of the possible
simply connected Lorentzian symmetric spaces together with
$(\R,-dt^2)$, the de Sitter, and the anti de Sitter spaces (see
\cite{CW}). Note that the curvature information of a plane wave is
contained in the profile $A$, this meaning that the only
non-vanishing component is
$$R_{uaub}=-A_{ab}(u).$$
The condition of being symmetric is then
$$\nabla R=0 \hspace{2mm}\Leftrightarrow\hspace{2mm} \partial_u A_{ab}=0,$$
which is obviously satisfied. As a symmetric space a Cahen-Wallach
space admits the homogeneous pseudo-Riemannian structure $S=0$.
The extra Killing vector field is $X=\partial_u$.

%
%
%
%
%

\bigskip

A singular scale-invariant homogeneous plane wave is a plane wave
with metric
$$g=dudv+\frac{B_{ab}}{u^2}x^ax^bdu^2+\sum_{a=1}^{n}(dx^a)^2,$$
where $(B_{ab})$ is a constant symmetric matrix. Unlike
Cahen-Wallach spaces these kind of plane waves are not
geodesically complete. Moreover, as their name suggests, these
spaces are homogeneous with extra Killing vector field
$X=u\partial_u-v\partial_v$, but not symmetric since the profile
$A(u)=B/u^2$ is not $u$-independent. They enjoy many properties.
For example, they have been found to occur universally as Pentose
Limits of space-time singularities \cite{BBOP}. Furthermore in
\cite{Mon} the following characterization is given.

\begin{theorem}
Let $(M,g)$ be a connected pseudo-Riemannian manifold of dimension
$n+2$ admitting a degenerate homogeneous pseudo-Riemannian
structure of liner type, i.e., $S_XY=g(X,Y)\xi-g(\xi,Y)X$ with
$g(\xi,\xi)=0$. Then $(M,g)$ is locally isometric to $\R^{n+2}$
with metric
$$ds^2=dudv+\frac{B_{ab}}{u^2}x^ax^bdu^2+\sum_{a=1}^{n}\varepsilon_a(dx^a)^2$$
for some symmetric matrix $(B_{ab})$ and $\varepsilon_a=\pm 1 ,
a=1,\ldots,n$.
\end{theorem}

Note that for Lorentzian signature this means that a manifold
admitting a degenerate homogeneous structure of linear type is
locally a singular scale-invariant homogeneous plane wave.
Conversely it is easy to see that every singular scale-invariant
homogeneous plane wave admits such a homogeneous structure with
$\xi=-\frac{1}{u}\partial_v$.

\subsection{The Lorentz-K\"ahler case}

By a Lorentz-K\"ahler manifold we understand a pseudo-K\"ahler
manifold of index $2$. In this subsection we will exhibit the
relation and similarities between Cahen-Wallach spaces and
singular scale-invariant homogeneous plane waves on one side and
some kind of (locally) homogeneous Lorentz-K\"ahler manifolds on
the other. Although, as far as the authors know, there is no
formal definition of a ``complex'' plane wave, this relation could
allow us to understand the latter spaces as a complex
generalization of the former, at least in the important
Lorentz-K\"ahler case, suggesting a starting point for a possible
definition of \textit{complex plane waves}.

\bigskip

Cahen-Wallach spaces are the model of symmetric Lorentzian plane
waves. Furthermore, as a wave, it is the twisted product of a
plane wave front and a two dimensional manifold containing time
and the direction of propagation. This two dimensional space gives
the real geometric information of the total manifold and in
particular it contains a null parallel vector field. For this
reason, in the Lorentz-K\"ahler case, we study symmetric manifolds
of complex dimension two (real dimension four) with a null
parallel one complex dimensional distribution. In the
classification of simply-connected indecomposable and not
irreducible pseudo-K\"ahlerian symmetric spaces of signature
$(2,2)$ given in \cite{Gal} (see also \cite{KO}), there is just
one possibility corresponding to such a situation. This is a
manifold with holonomy algebra
$$\f{hol}^{\gamma_1=0, \gamma_2=0}_{n=0} = \R p_1\wedge p_2=\R
\left(\begin{array}{cccr}
0 & 0 & 0 & -1\\
0 & 0 & 1 & 0\\
0 & 0 & 0 & 0\\
0 & 0 & 0 & 0
\end{array}\right)$$
and curvature $R_{\lambda_5=1}$ or $-R_{\lambda_5=1}$ in the
notation of \cite{Gal} which we explain now. Here the tangent
space has been identified with $\R^{2,2}=\C^{1,1}$ with a basis
$\{p_1,p_2,q_1,q_2\}$ with respect to which the metric and the
complex structure take form
$$\left(\begin{array}{cccr}
0 & 0 & 1 & 0\\
0 & 0 & 0 & 1\\
1 & 0 & 0 & 0\\
0 & 1 & 0 & 0
\end{array}\right)\hspace{1em}\text{and} \hspace{1em}
\left(\begin{array}{cccr}
0 & -1 & 0 & 0\\
1 & 0 & 0 & 0\\
0 & 0 & 0 & -1\\
0 & 0 & 1 & 0
\end{array}\right)
$$
respectively. The curvature $R_{\lambda_5=1}$ stands for the
curvature operator $R\in S^2(\R^{2,2}\wedge\R^{2,2})$ that sends
all elements of the basis to zero except for $R(q_1\wedge
q_2)=p_1\wedge p_2$.

Now, in order to get the K\"ahler-Lorentz analog of Cahen-Wallach
spaces, we add a $n$ complex dimensional plane wave front to this
four dimensional manifold. This is done by considering the
holonomy algebra $\f{hol}^{\gamma_1=0,
\gamma_2=0}_{n=0}\oplus\{\mathrm{Id_{2n}}\}$ acting in
$\R^{2n+4}=\R^{2,2}\oplus \R^{2n}$ through the action of
$\f{hol}^{\gamma_1=0, \gamma_2=0}_{n=0}$ in $\R^{2,2}$ and the
trivial action in $\R^{2n}$.

\begin{proposition}
Let $(M,g,J)$ be a (locally) symmetric Lorentz-K\"ahler manifold
of dimension $2n+4$, $n\geq 0$, with holonomy
$\f{hol}^{\gamma_1=0, \gamma_2=0}_{n=0}\oplus\{\mathrm{Id_{2n}}\}$
acting on $T_xM\simeq \R^{2n+4}$, $x\in M$, as explained above.
The metric $g$ is locally of the form
\begin{multline}\label{metric1}
g  = dw^1dz^1+dw^2dz^2+b(dw^1dw^1+dw^2dw^2)
+\sum_{a=1}^nr_a(dx^adw^1+dy^adw^2)\\+\sum_{a=1}^ns_a(dx^adw^2-dy^adw^1)+\sum_{a=1}^n(dx^adx^a+dy^ady^a),
\end{multline}
where the functions $b,r_a,s_a$, $a=1,\ldots,n$, only depend on
$w^1$ and $w^2$ and satisfy
$$\dfrac{\partial s_a}{\partial
w^1}=\dfrac{\partial r_a}{\partial w^2},\qquad \dfrac{\partial
s_a}{\partial w^2}=-\dfrac{\partial r_a}{\partial
w^1}\vspace{2mm},\qquad \Delta b=b_0,\hspace{1em}
b_0\in\bb{R}-\{0\}.$$
\end{proposition}

\begin{proof}
First note that since $\nabla R=0$ the holonomy algebra at a point
$x$ is generated by the elements $R_{XY}$, $X,Y\in T_xM$. Let
$\{e_1,\ldots,e_n,Je_1,\ldots,$ $Je_n\}$ be an orthonormal basis
of $\R^{2n}$, since $p_1,p_2,e_i,Je_i$, $i=1,\ldots,n$, are
invariant by holonomy, we can extend them by parallel transport
defining parallel vector fields $Z,JZ,E_i,JE_i$. Let now
$\alpha^1=g(\cdot,Z)$ and $\alpha^2=-\alpha^1\circ J$, consider
the complex form $\alpha=\alpha^1+i\alpha^2$. Since $\nabla
Z=0=\nabla JZ$, we have $\nabla \alpha=0$, hence in particular
$\alpha$ is holomorphic and closed. This means that locally there
is a holomorphic function $w:U\to \C$ such that $dw=\alpha$. Since
$dw$ is non-zero at some point and it is parallel, we have that
$dw$ is never zero. Hence if the set $w^{-1}(\lambda)$,
$\lambda\in\C$, is non-empty it defines a complex hypersurface in
$U$. Note that since $g(Z,Z)=0=g(Z,E_i)=g(Z,JE_i)$ the vector
fields $Z,JZ,E_i,JE_i$ are always tangent to the hypersurfaces
$w^{-1}(\lambda)$. Therefore we can take coordinates
$\{w^1,w^2,z^1,z^2,x^i,y^i\}$ such that $w=w^1+iw^2$,
$\partial_{z^1}=Z$, $\partial{z^2}=JZ$, $\partial_{x^i}=E_i$ and
$\partial_{y^i}=JE_i$. With respect to this coordinates the metric
is
\begin{multline*}
g  = dw^1dz^1+dw^2dz^2+b(dw^1dw^1+dw^2dw^2)
+\sum_{a=1}^nr_a(dx^adw^1+dy^adw^2)\\+\sum_{a=1}^ns_a(dx^adw^2-dy^adw^1)+\sum_{a=1}^n(dx^adx^a+dy^ady^a),
\end{multline*}
where the functions $b,r_a,s_a$ only depend on $w^1$ and $w^2$.
Now imposing $\nabla\partial_{x^a}=0=\nabla\partial_{y^a}$ we
obtain
\begin{eqnarray}\label{Cauchy-Riemann2}
\frac{\partial s_a}{\partial w^1}&=&\frac{\partial r_a}{\partial
w^2}\nonumber\\
\frac{\partial s_a}{\partial w^2}&=&-\frac{\partial r_a}{\partial
w^1}.
\end{eqnarray}
In addition, the only non-zero element of the curvature tensor is
$$R_{\partial{w^1}\partial{w^2}\partial{w^1}\partial{w^2}}=\frac{1}{2}\Delta
b,$$
where $\Delta$ stands for the Laplace operator with respect
to the variables $(w^1,w^2)$. The condition of being (locally)
symmetric is then
$$\nabla R=0 \hspace{2mm}\Leftrightarrow\hspace{2mm} \Delta b=b_0,$$
for $b_0\in\R-\{0\}$.
\end{proof}

\bigskip

In view of this Proposition we consider the pseudo-K\"ahler
manifold $(\C^{2+n},g)$, with $g$ given by \eqref{metric1}, as a
natural Lorentz-K\"ahler equivalent to Cahen-Wallach spaces. Note
that the Laplacian condition admits solutions with singularities.
As Cahen-Wallach spaces are simply-connected, we only consider
solutions $b$ defined on the whole $\C^{2+n}$. For these $b$, it
is easy to check that $(\C^{2+n},g)$ is geodesically complete.

\bigskip

We now study Lorentzian singular scale-invariant homogeneous plane
waves. As these manifolds are characterized by degenerate
homogeneous structure tensors of linear type (see \S \ref{LKc}),
from Theorem \ref{prop princ} above the natural equivalent to this
spaces are Lorentz-K\"ahler manifolds with strongly degenerate
homogeneous pseudo-K\"ahler structure tensors of linear type, and
more precisely, the space $(\C^{n+2}-\{w^1=w^2=0\},g)$ with $g$
given by \eqref{metric2}. Moreover, the local expression of the
metric \eqref{metric2} given in Theorem \ref{prop princ}
(restricted to signature $(2,2+2n)$) and the metric
\eqref{metric1} are the same except for the function $b$, which
has a different Laplacian in each case. As a straight forward
computation shows, the curvature tensor of both metrics
\eqref{metric1} and \eqref{metric2} is
$$R=\frac{1}{2}\Delta b(dw^1\wedge dw^2)\otimes(dw^1\wedge
dw^2),$$ so all the curvature information is contained in the
Laplacian of the function $b$. For this reason, analogously to
Lorentz plane waves, we call $\Delta b$ the profile of the metric.
It is worth noting that in the Lorentz case one goes from
Cahen-Wallach spaces to singular scale-invariant homogeneous plane
waves by making the profile be singular with a term $1/u^2$. Doing
so, the space is no longer geodesically complete and a
cosmological singularity at $\{u=0\}$ is created. In the same way,
in the Lorentz-K\"ahler case one goes from metric \eqref{metric1}
to \eqref{metric2} by making the profile be singular with a term
$1/\left((w^1)^2+(w^2)^2\right)^2$ and again one transforms a
geodesically complete space to a geodesically uncomplete space,
and a cosmological singularity at $\{w^1=w^2=0\}$ is created. This
exhibits a close relation between this two couples of spaces.

\bigskip

\begin{center}
\begin{tabular}{|c|c|c|}
\hline
\multirow{2}{*}{} & \T Symmetric  & \B Strongly Deg. homog. \\
 & space & of linear type\\[0.5ex]
\hline

\multirow{5}{*}{Lorentz} &  \T Cahen-Wallach & \T Singular s.-i.
homog.\\
       & spaces & plane wave \\
       & & \\
                & \footnotesize{Profile: $A(u)=A(const.)$} & \footnotesize{Profile: $A(u)=A/u^2$}\\

                & \footnotesize{Geodesically complete} & \footnotesize{Geodesically
                uncomplete}\\
                                [0.5ex]
\hline
\multirow{5}{*}{Lorentz-K\"ahler} & \T $\C^{2+n}$ with metric  & \T $\C^{2+n}-\{w=0\}$\\ & \eqref{metric1} & with metric \eqref{metric2}\\
 & & \\
  & \footnotesize{Profile: $\Delta b=b_0(const.)$} & \footnotesize{Profile: $\Delta b=b_0/||w||^4$}\\
 & \footnotesize{Geodesically complete} & \footnotesize{Geodesically
 uncomplete}\\
  [0.5ex]
\hline
\end{tabular}
\end{center}

\bigskip

Finally, if we now write the expression of the metrics
\eqref{metric1} and \eqref{metric2} in complex notation (see
Corollary \ref{metrica comp 1} above) we get
\begin{eqnarray*}
h & = & \frac{1}{2}\left( dw\otimes d\overline{z}+dz\otimes
d\overline{w}\right)+\frac{b}{2}dw\otimes
d\overline{w}+\sum_{a=1}^n\frac{h_a}{2}dw\otimes
d\overline{z^a}\nonumber\\
& & +\sum_{a=1}^n\frac{\overline{h_a}}{2}dz^a\otimes
d\overline{w}+\frac{1}{2}\sum_{a=1}^n dz^a\otimes d\overline{z^a}.
\end{eqnarray*}
This expressions obey the general formula of a \textit{pp-wave}
(plane wave front with parallel rays) (see e.g. \cite{Bic},
\cite{EK}) but written with complex coordinates instead of real
coordinates. This kind of Lorentz manifolds include plane waves
and are related to (gravitational) radiation propagating at the
speed of light. Plane waves are exact solutions of Einstein's
field equations and metrics \eqref{metric1} and \eqref{metric1}
are Ricci-flat, so they solve vacuum Einstein's field equations.
Finally, it is worth noting that metrics \eqref{metric1} and
\eqref{metric2} are VSI (see Proposition \ref{proposition VSI}) a
common property of all plane waves.

\section{The pseudo-hyper-K\"ahler and pseudo-quaternion K\"ahler case}

During this section $\mathrm{dim}(M)=4n\geq 8$ is assumed. We
shall study strongly degenerate homogeneous structures of linear
type in the pseudo-hyper-K\"ahler and the pseudo-quaternion
K\"ahler cases.

\begin{definition}
Let $(M,g)$ be a pseudo-Riemannian manifold. A pseudo-quaternionic
Hermitian structure is a $3$-rank subbundle $\upsilon^3\subset
\f{so}(TM)$ with a local basis $J_1,J_2,J_3$ satisfying
$$J_1^2=J_2^2=J_3^2=-1,\qquad J_1J_2=J_3.$$
\end{definition}

This means that at every point $p\in M$ there is a subalgebra
$\upsilon^3_p\subset\f{so}(T_pM)$ isomorphic to the imaginary
quaternions, and in particular $g$ has signature $(4p,4q)$.

\begin{definition}
A pseudo-Riemannian manifold $(M,g)$ is called pseudo-quaternion
K\"ahler if it admits a parallel pseudo-quaternionic Hermitian
structure with respect to the Levi-Civita connection, or
equivalently if the holonomy group of the Levi-Civita connection
is contained in $Sp(p,q)Sp(1)$.
\end{definition}


Let $J_1,J_2,J_3$ be a local basis of $\upsilon^3$, and
$\omega_a=g(\cdot,J_a\cdot)$, $a=1,2,3$. The $4$-form
$$\Omega=\omega_1\wedge\omega_1+\omega_2\wedge\omega_2+\omega_3\wedge\omega_3$$
is independent of the choice of basis and hence it is globally
defined. A  pseudo-quaternionic Hermitian manifold
$(M,g,\upsilon^3)$ is pseudo-quaternion K\"ahler if and only if
$\Omega$ is parallel with respect to the Levi-Civita connection
(cf. \cite{AC}).

\begin{definition} An pseudo-quaternion K\"ahler
manifold $(M,g,\upsilon^3)$ is called a homogeneous
pseudo-quaternion K\"ahler manifold if there is a connected Lie
group $G$ of isometries acting transitively on $M$ and preserving
$\upsilon^3$. $(M,g,\upsilon^3)$ is called a reductive homogeneous
pseudo-quaternion K\"ahler manifold if the Lie algebra $\f{g}$ of
$G$ can be decomposed as $\f{g}=\f{h}\oplus\f{m}$ with
$$[\f{h},\f{h}]\subset \f{h},\qquad[\f{h},\f{m}]\subset \f{m}.$$
\end{definition}

As a corollary of Kiri\v{c}enko's Theorem \cite{Kir} we have

\begin{theorem}
A connected, simply connected and (geodesically) complete
pseudo-quaternion K\"ahler manifold $(M,g,\upsilon^3)$ is
reductive homogeneous if and only if it admits a linear connection
$\wnabla$ satisfying \begin{equation}\label{AS qK}\wnabla
g=0,\qquad \wnabla R=0,\qquad \wnabla S=0,\qquad
\wnabla\Omega=0,\end{equation} where $S=\nabla-\wnabla$, $\nabla$
is the Levi-Civita connection, $R$ is the curvature tensor of
$\nabla$, and $\Omega$ is the canonical $4$-form associated to
$\upsilon^3$.
\end{theorem}

A tensor field $S$ satisfying the previous equations is called a
\emph{homogeneous pseudo-quaternion K\"ahler structure}. The
classification of such structures was obtained in \cite{BGO},
resulting five primitive classes $\QK_1,\QK_2$,
$\QK_3,\QK_4,\QK_5$. Among them $\QK_1,\QK_2,\QK_3$ have dimension
growing linearly with respect to the dimension of $M$. Hence


\begin{definition}
A homogeneous pseudo-quaternion K\"ahler structure $S$ is called
of linear type if it belongs to the class $\QK_1+\QK_2+\QK_3$.
\end{definition}

The local expression of $S\in\QK_1+\QK_2+\QK_3$ is
\begin{multline}\label{struct linear type
QT}S_XY=g(X,Y)\xi-g(Y,\xi)X+\sum_{a=1}^3\left(g(J_aY,\xi)J_aX-g(X,J_aY)J_a\xi\right)\\
+\sum_{a=1}^3g(X,\zeta^a)J_aY,
\end{multline}
where $\xi$ and $\zeta^a$, $a=1,2,3$, are vector fields. We then
give the following further definition.

\begin{definition}
A homogeneous pseudo-quaternion K\"ahler structure of linear type
given by formula (\ref{struct linear type QT}) is called strongly
degenerate if $\xi\neq 0$, $g(\xi,\xi)=0$ and $\zeta^a=0$ for
$a=1,2,3$.
\end{definition}

\begin{proposition}\label{prop qK}
Let $(M,g,\upsilon^3)$ be a pseudo-quaternion K\"ahler manifold
admitting a strongly degenerate homogeneous pseudo-quaternion
K\"ahler structure of linear type. Then $(M,g,\upsilon^3)$ is
flat.
\end{proposition}

\begin{proof}
Let $S$ be a strongly degenerate homogeneous pseudo-quaternion
K\"ahler structure of linear type on $(M,g,\upsilon^3)$. From
(\ref{struct linear type QT}) and the third equation in (\ref{AS
qK}) we have
$$\nabla_X\xi=S_X\xi=g(X,\xi)\xi-\sum_{a=1}^ng(X,J_a\xi)J_a\xi,$$
and
$$\nabla_XJ_a\xi  =
\widetilde{\tau}^c(X)J_b\xi-\widetilde{\tau}^b(X)J_c\xi+g(X,J_a\xi)\xi-\sum_{d=1}^3g(X,J_dJ_a\xi)J_d\xi$$
for certain $1$-forms
$\widetilde{\tau}^1,\widetilde{\tau}^2,\widetilde{\tau}^3$ defined
by the parallel property of $\upsilon ^3$, where $(a,b,c)$ is any
cyclic permutation of $(1,2,3)$. Using this formulas and after a
long calculation one obtains
$$R_{XY}\xi=\nabla_{[X,Y]}\xi-[\nabla_X,\nabla_Y]\xi=0.$$
Since $(M,g,\upsilon^3)$ is Einstein we have
$$0=r(X,\xi)=\nu_qg(X,\xi),\qquad X\in\f{X}(M),$$
where $\nu_q$ is one-quarter of the reduced scalar curvature.
Supposing that $\xi\neq 0$ this implies that $\nu_q=0$ and hence
$(M,g,\upsilon^3)$ is Ricci-flat. Therefore the manifold is
locally hyper-K\"ahler and the curvature $R$ is of type
$\f{sp}(p,q)$, i.e., $R_{XYJ_aZW}+R_{XYZJ_aW}=0$ for $a=1,2,3$.

Now, the second equation in (\ref{AS qK}) reads
$$\left(\nabla_XR\right)_{YZWU}=-R_{S_XYZWU}-R_{YS_XZWU}-R_{YZS_XWU}-R_{YZWS_XU},$$
so taking the cyclic sum in $X,Y,Z$ and applying Bianchi
identities, after some computations we get
\begin{eqnarray*}
0 & = &
\Cyclic{\fdS}{\SXYZ}\left\{g(\xi,Y)R_{XZWU}-\sum_{a=1}^3g(\xi,J_aY)R_{J_aXZWU}\right.\\
 & + & g(\xi,Z)R_{YXWU}-\sum_{a=1}^3g(\xi,J_aZ)R_{YJ_aXWU}\\
 & + & g(\xi,W)R_{YZXU}-\sum_{a=1}^3g(\xi,J_aW)R_{YZJ_aXU}\\
 & + & \left.
 g(\xi,U)R_{YZWX}-\sum_{a=1}^3g(\xi,J_aU)R_{YZWJ_aX}\right\}\\
 & = & 2\Cyclic{\fdS}{\SXYZ}g(X,\xi)R_{ZYWU}.
\end{eqnarray*}
This means that $\theta\wedge R_{WU}=0$ with $\theta = \xi
^{\flat}$. But since $R$ is of type $\f{sp}(p,q)$ we also have
$$\left\{\begin{array}{r}
\left(\theta\circ J_1\right)\wedge R_{WU}=0\\
\left(\theta\circ J_2\right)\wedge R_{WU}=0\\
\left(\theta\circ J_3\right)\wedge R_{WU}=0. \end{array}\right.$$
It is easy to see that these equations force $R_{WU}=0$, and hence
$(M,g,\upsilon^3)$ must be flat.
\end{proof}

\bigskip

\begin{definition}
Let $(M,g)$ be a pseudo-Riemannian manifold. A
pseudo-hyper-Hermitian structure is a subset
$\{J_1,J_2,J_3\}\subset \f{so}(TM)$ satisfying
$$J_1^2=J_2^2=J_3^2=-1,\qquad J_1J_2=J_3.$$ In particular $g$ has signature $(4p,4q)$.
\end{definition}

\begin{definition}
A pseudo-Riemannian manifold $(M,g)$ is called
pseudo-hyper-K\"ahler if it admits a parallel
pseudo-hyper-Hermitian structure with respect to the Levi-Civita
connection, or equivalently if the holonomy group of the
Levi-Civita connection is contained in $Sp(p,q)$.
\end{definition}

\begin{definition}
A pseudo-hyper-K\"ahler manifold $(M,g,J_1,J_2,J_3)$ is called a
homogeneous pseudo-hyper-K\"ahler manifold if it admits a
connected Lie group $G$ of isometries acting transitively on $M$
and preserving $J_a$, $a=1,\ldots,3$. $(M,g,J_1,J_2,J_3)$ is
called a reductive homogeneous pseudo-hyper-K\"ahler manifold if
the Lie algebra $\f{g}$ of $G$ can be decomposed as
$\f{g}=\f{h}\oplus\f{m}$ with
$$[\f{h},\f{h}]\subset \f{h},\qquad[\f{h},\f{m}]\subset \f{m}.$$
\end{definition}

It is worth noting that, unlike the Riemannian setting, a
homogeneous and Ricci-flat pseudo-Riemannian manifold is not
necessarily flat. As a corollary of Kiri\v{c}enko's Theorem
\cite{Kir} we have again

\begin{theorem}
A connected, simply connected and (geodesically) complete
pseudo-hyper-K\"ahler manifold $(M,g,J_a)$ is reductive
homogeneous if and only if it admits a linear connection $\wnabla$
satisfying
\begin{equation}\label{AS hK}\wnabla g=0,\qquad \wnabla R=0,\qquad
\wnabla S=0,\qquad \wnabla
J_a=0,\hspace{2mm}a=1,\ldots,3\end{equation} where
$S=\nabla-\wnabla$, $\nabla$ is the Levi-Civita connection and $R$
is the curvature tensor of $\nabla$.
\end{theorem}

A tensor field $S$ satisfying the previous equations is called a
\emph{homogeneous pseudo-hyper-K\"ahler structure}. The
classification of these structures is obtained in \cite{CL},
resulting three primitive classes $\HK_1,\HK_2$, $\HK_3$. Among
them only $\HK_1$ has dimension growing linearly with respect to
the dimension of $M$, hence we call a homogeneous
pseudo-hyper-K\"ahler structure of \emph{linear type} if it
belongs to the class $\HK_1$. The expression of these tensors is
$$S_XY=g(X,Y)\xi-g(Y,\xi)X+\sum_{a=1}^3\left(g(J_aY,\xi)J_aX-g(X,J_aY)J_a\xi\right),$$
where $\xi$ is a vector field. Analogously to the previous cases,
$S$ is called strongly degenerate if $\xi\neq 0$ is isotropic.

Let now $(M,g,J_a)$ be a pseudo-hyper-K\"ahler manifold admitting
a strongly degenerate homogeneous pseudo-hyper-K\"ahler structure
of linear type. Repeating exactly the same computations as in the
pseudo-quaternion K\"ahler case, but knowing a priori that
$\nu_q=0$, we arrive to the conclusion that $(M,g,J_a)$ must be
flat.

\begin{remark}
Note that in the pseudo-K\"ahler case we have proved (see
Proposition \ref{prop hol}) that admitting a strongly degenerate
structure of linear type automatically implies that the manifold
has an integrable $SU(p,q)$ structure. Moreover
$SU(p,q)$-homogeneous structures of linear type have the same
expression as strongly degenerate pseudo-K\"ahler structures. So
the $SU(p,q)$ case is already done. Homogeneous manifolds with
other geometric structures, such as $G_{2(2)}^*$ or $Spin(4,3)$,
can also admit structures of linear type, namely there is a
submodule of the space of homogeneous structures with dimension
growing linearly with the dimension of the manifold. The study of
such structures is open.
\end{remark}

\section*{\small{Acknowledgements}}

\footnotesize{The authors are deeply indebted to Prof. Andrew
Swann and Prof. P.M. Gadea for useful conversations about the
topics of this paper.

\noindent This work has been partially funded by MINECO (Spain)
under project MTM2011-22528.}

\end{document}